\input ./style/arxiv-general.cfg
\documentclass[aop,MSNbibl,dvips]{arximspdf}
\makeatletter
   \@ifpackageloaded{graphicx}{}{\usepackage{graphicx}}
\makeatother

\doi{10.1214/13-AOP904} 
\volume{43}
\issue{3}
\pubyear{2015}
\firstpage{1493}
\lastpage{1534}

\makeatletter
\newcommand{\rrvert}{\vert}
\newcommand{\llvert}{\vert}
\newcommand{\xrightarrow}[1]{\stackrel{#1}{\to}}
\newcommand{\xxrightarrow}[1]{\stackrel{#1}{\longrightarrow}}
\newtheorem{theorem}{Theorem}[section]
\newtheorem{lemma}[theorem]{Lemma}
\newproclaim{notation}[theorem]{Notation}
\newtheorem{proposition}[theorem]{Proposition}
\newproclaim{remark}[theorem]{Remark}
\makeatother

\begin{document}
\begin{frontmatter}

\title{On \textit{L}\tsup{2} modulus of continuity of Brownian local
times and Riesz potentials}
\runtitle{Modulus of continuity of Brownian local times}

\begin{aug}
\author[A]{\fnms{Aur\'elien} \snm{Deya}\corref{}\ead[label=e1]{aurelien.deya@univ-lorraine.fr}},
\author[B]{\fnms{David} \snm{Nualart}\thanksref{T2}\ead[label=e2]{nualart@math.ku.edu}}
\and
\author[A]{\fnms{Samy} \snm{Tindel}\thanksref{T3}\ead[label=e3]{samy.tindel@univ-lorraine.fr}}
\runauthor{A. Deya, D. Nualart and S. Tindel}
\affiliation{Universit\'e de Lorraine, University of Kansas and
Universit\'e de Lorraine}
\address[A]{A. Deya\\
S. Tindel\\
Institut {\'E}lie Cartan\\
Universit\'e de Lorraine\\
B.P. 239\\
54506 Vand{\oe}uvre-l{\`e}s-Nancy\\
France\\
\printead{e1}\\
\phantom{E-mail:\ }\printead*{e3}} 
\address[B]{D. Nualart\\
Department of Mathematics\\
University of Kansas\\
405 Snow Hall\\
Lawrence, Kansas\\
USA\\
\printead{e2}}
\end{aug}
\thankstext{T2}{Supported by NSF Grant DMS-12-08625.}
\thankstext{T3}{Member of the BIGS (Biology, Genetics and Statistics)
team at INRIA Nancy.}

\received{\smonth{4} \syear{2013}}
\revised{\smonth{11} \syear{2013}}

%
\begin{abstract}
This article is concerned with modulus of continuity of Brownian local
times. Specifically, we focus on three closely related problems: (a)
Limit theorem for a Brownian modulus of continuity involving Riesz
potentials, where the limit law is an intricate Gaussian mixture. (b)
Central limit theorems for the projections of $L^{2}$ modulus of
continuity for a one-dimensional Brownian motion. (c)~Extension of the
second result to a two-dimensional Brownian motion. Our proofs rely on
a combination of stochastic calculus and Malliavin calculus tools, plus
a thorough analysis of singular integrals.
\end{abstract}

%
\begin{keyword}[class=AMS]
\kwd{60G15}
\kwd{60H07}
\kwd{60H10}
\kwd{65C30}
\end{keyword}
\begin{keyword}
\kwd{Brownian motion}
\kwd{local time}
\kwd{Malliavin calculus}
\end{keyword}

\end{frontmatter}

\section{Introduction}\label{sec1}
Let $\{B_t, 0\leq t\leq1\}$ be a standard linear Brownian motion
defined on some complete probability space $(\Omega,\mathcal
{F},\mathbf
{P})$. In the sequel, we denote by $L_{t}(x)$ the local time of $B$ at
a given point $x\in\mathbb R$, defined for $t\in[0,1]$. A nice
combination of
stochastic calculus, stochastic analysis and evaluation of
singularities associated with heat kernels have recently led to a
number of interesting limit theorems for quantities related to the
family $\{L_{t}(x); t\in[0,1], x\in\mathbb R\}$. Let us quote,
for instance,
the use of Malliavin and stochastic calculus tools in order to get
suitably normalized limits for $L^{2}$ modulus of continuity (see \cite
{HN09,Ro11}) or third moment in space (cf. \cite{HN10}) of Brownian
local time. Malliavin calculus tools have also been essential in order
to generalize the notion of self-intersection local time \cite
{HN05,HNS08} and to obtain central limit theorems for additive
functionals~\cite{HNX11} of fractional Brownian motion.

The current article proposes to take another step into the
relationships between Brownian local time and stochastic analysis.
Specifically, we shall handle the following problems:
\begin{longlist}[(2)]
\item[(1)] One of the motivation alluded to in \cite{Ro11} for the
renormalization of $L^{2}$ modulus of continuity of local times comes
from the study of the Hamiltonian
%
\begin{eqnarray}\label{eqhamiltonian-mod-cty-loc-time}
H_t^{h}(B) &=& \int_{\mathbb R}
\bigl[L_{t}(x+h) - L_{t}(x) \bigr]^{2} \,dx
\nonumber\\[-8pt]\\[-8pt]
&=&\int
_{\mathbb R} \biggl[\int_{0}^{t}
\bigl(\delta_{x+h}(B_u) - \delta_{x}(B_u)
\bigr) \,du \biggr]^{2} \,dx,\nonumber
\end{eqnarray}
which is involved in the definition of some nonfolding polymers.
However, one might wish to consider a slightly weaker repelling
self-interaction of the polymer by introducing the following family of
Hamiltonians indexed by $\gamma\in(0,1)$:
%
\begin{equation}
\label{eqdef-hamiltonian-gamma} H_t^{h,\gamma}(B) = \int_{\mathbb R}
\biggl[\int_{0}^{t} \bigl(|B_v+x+h|^{-\gamma}
- |B_v+x|^{-\gamma} \bigr)\,dv \biggr]^{2} \,dx.
\end{equation}
For this modified Hamiltonian, we shall prove the following limiting theorem:

%
\begin{theorem}\label{thmlim-hamiltonian-gamma}
Consider $\gamma\in(3/4,1)$ and the family of Hamiltonians $\{
H_t^{h,\gamma
}(B); t\in[0,1]\}$ defined by~(\ref{eqdef-hamiltonian-gamma}). Then
one has, as $h$ tends to zero,
%
\begin{equation}
\label{eqlim-hamiltonian-gamma} \frac{H^{h,\gamma}(B) - \mathbf E [H^{h,\gamma}(B)
]}{c_\gamma
h^{7/2-2\gamma}} \xrightarrow{(d)} W_{\alpha}
\end{equation}
in the space $\mathcal{C}([0,1];\mathbb R)$ of real continuous
functions on
$[0,1]$. In relation~(\ref{eqlim-hamiltonian-gamma}), $c_\gamma$~stands
for a deterministic positive constant depending only on $\gamma$, $W$
is a
standard Brownian motion independent of $B$ and $\alpha$ is the
self-intersection local time of $B$, that is (formally),
%
\begin{equation}
\label{eqdef-alpha} \alpha_t:=\int_0^t
\,dv \int_0^v \,du \delta_0(B_v-B_u),
\end{equation}
where $\delta_0$ is the Dirac delta function concentrated at $0$.
\end{theorem}

Theorem~\ref{thmlim-hamiltonian-gamma} turns out to be interesting for
several reasons:
\begin{itemize}
\item The Hamiltonian $H^{h,\gamma}(B)$ quantifies a weak self-interaction of
the Brownian path, detecting if the path self intersects (products of
the form $|B_{v_{1}}+x|^{-\gamma}|B_{v_{2}}+x|^{-\gamma}$) or has a
fold with amplitude $h$ (products of the form $|B_{v_{1}}+x+h|^{-\gamma
}|B_{v_{2}}+x|^{-\gamma}$). It can thus be related to the polymer model
studied in \cite{HKK}, where a discrete time random walk $S_n$ on
$\mathbb{Z}$ is weighted according to the following Hamiltonian:
\[
H_n= \sum_{i,j=1}^n
\mathbf{1}_{\{S_i=S_j\}}-\sum_{i,j=1}^n
\mathbf {1}_{\{|S_i-S_j|=1\}}. %
\]
This relation was also the motivation behind the central limit theorem
given in~\cite{Ro11}, and other physically relevant models for
self-interacting continuous paths include Brownian filaments (see \cite
{FG02} for a detailed definition of these objects), motivated by
turbulent fluids. We thus hope that the scaling limit for our quantity
$H^{h,\gamma}(B)$ can shed some light on the aforementioned models.

\item
Theorem~\ref{thmlim-hamiltonian-gamma} also exhibits an interesting
phenomenon in terms of limiting behavior. Indeed, the reader can easily
observe that the limiting process in the right-hand side of (\ref
{eqlim-hamiltonian-gamma}) does not depend on the parameter $\gamma$ in
$(3/4,1)$, the only difference lying in the normalizing quantity
$c_\gamma
h^{7/2-2\gamma}$. Furthermore, it was shown in \cite{HN09,Ro11} that
relation~(\ref{eqlim-hamiltonian-gamma}) still holds true in the
limiting case $\gamma=1$. This means that the process $W_{\alpha}$, which
can be seen as a Gaussian mixture, might also be considered as a rather
canonical object.

\item
At a methodological level, our proof of Theorem~\ref{thmlim-hamiltonian-gamma} is another example of the interest of
stochastic calculus techniques with respect to the method of moments in
this context. We should compare our methodology, for example, to the
computationally demanding paper \cite{CLMR}. The advantage of
stochastic calculus methods had already been highlighted in \cite
{HN09,Ro11}, but our proof combines this approach with an extensive use
of Fourier analysis techniques.
\end{itemize}

\item[(2)] Go back now to the Hamiltonian $H_t^{h}(B)$ defined by
(\ref{eqhamiltonian-mod-cty-loc-time}) and related to $L^2$ modulus of
continuity of the Brownian local time. As mentioned above, it has been
shown in~\cite{HN09,Ro11} that $h^{-3/2}(H^{h}(B) - \mathbf E[
H^{h}(B) ])$
converges in law to $c_{1}W_{\alpha}$ for a universal constant $c_1$, that
is, relation~(\ref{eqlim-hamiltonian-gamma}) is still formally
satisfied for $\gamma=1$. This noncentral limit theorem indicates that an
interesting phenomenon might occur as far as limiting behavior of the
renormalized quantity $h^{-3/2}(H^{h}(B) - \mathbf E[ H^{h}(B) ])$ on
chaoses is concerned.
We shall specify this with the following result:
\end{longlist}

%
\begin{theorem}\label{thmcvgce-L2-modulus-chaos}
Let $\{H^{h}_{t}(B); t\in[0,1]\}$ be the process defined by (\ref
{eqhamiltonian-mod-cty-loc-time}). For a given random variable $F\in
L^{2}(\Omega)$ and for all $n\geq0$, we set $J_n(F)$ for the projection
of $F$ on the $n$th chaos of $B$, and subsequently
define $X^{n,h}_{t}\equiv J_{n}(H^{h}_{t}(B))$. Then:
\begin{longlist}[(iii)]
\item[(i)] For all $m\ge0$ and all $t\in[0,1]$, $h>0$ we have
$X^{2m+1,h}_{t}=0$.

\item[(ii)] For all $m\ge1$, we have as $h$ tends to zero,
\[
\frac{X^{2m,h} }{h^{2} [\ln(1/h)]^{1/2}} \xrightarrow{(d)} \sigma_{m} W\qquad\mbox{with }
\sigma_{m}^{2}= \frac{c (2m-2)!}{2^{2m} [(m-1)!]^2},
\]
where $W$ stands for a Brownian motion independent of $B$ and where the
convergence takes place in the space $\mathcal{C}([0,1];\mathbb R)$ of real
continuous functions on $[0,1]$.

\item[(iii)]
In particular, the series $\sum_{m\ge1} \sigma_{m}^{2}$ is divergent.
\end{longlist}
\end{theorem}

Putting together the results of \cite{HN09} and our Theorem~\ref{thmcvgce-L2-modulus-chaos}, we thus get the following picture: on the
one hand, one can renormalize the process $H^{h}(B)$ by $h^{3/2}$
in order to get a limit that is a mixture of Gaussian processes (a
noncentral type limit theorem). On the other hand, each projection
$J_n(H^{h}(B))$ can be properly renormalized (by $h^{2} [\ln
(1/h)]^{1/2}$) so as to obtain a limiting object that is a weighted
Brownian motion (corresponding to a central limit theorem).
Nevertheless, the sum of the weights $\sigma_{n}^{2}$ obtained by
projection is divergent. To the best of our knowledge, this interesting
limiting behavior is exhibited here for the first time. Note that it
contrasts, for instance, with the situation described in \cite{HN05}, Theorem~3 (and more specifically in the applications of this result),
where under appropriate variance assumptions, the normal convergence in
each chaos guarantees the normal convergence of the sum.

\begin{longlist}[(3)]
\item[(3)]
Finally, we consider a suitable generalization of Theorem~\ref{thmcvgce-L2-modulus-chaos} to a two-dimensional Brownian motion $B$.
Namely, we shall obtain the following convergence result.
\end{longlist}

%
\begin{theorem}\label{thmcvgce-L2-modulus-chaos-2d}
Let $\{H^{h}_{t}(B); t\in[0,1]\}$ be the process defined by (\ref
{eqhamiltonian-mod-cty-loc-time}), for a two-dimensional Brownian
motion $B$. Like in Theorem~\ref{thmcvgce-L2-modulus-chaos}, we define
$X^{n,h}_{t}$ as the projection on the $n$th chaos of $H^{h}_{t}(B)$.
Then the assertions \textup{(i)--(iii)} of Theorem~\ref{thmcvgce-L2-modulus-chaos} are still valid in this situation, with
\textup{(ii)} replaced with the following statement:
\begin{longlist}[(ii-2d)]
\item[(ii-2d)] For all $m\ge1$, we have as $h$ tends to zero,
\[
\frac{X^{2m,h} }{|h|} \xrightarrow{(d)} \sigma_{m} W,
\]
where $W$ stands for a linear Brownian motion independent of $B$, where
the exact expression of $\sigma_{m}$ will be specified at Section~\ref
{secbehavior-variance-2d} and where the convergence takes place in the
space $\mathcal{C}([0,1];\mathbb R)$ of $\mathbb R$-valued continuous
functions on $[0,1]$.
\end{longlist}
\end{theorem}

It is worthwhile noting that the equivalent of the main result of \cite
{HN10}, namely the convergence in law of a suitably renormalized
version of $H^h_t(B)$, is not available in the two-dimensional case.
Indeed, one can formally show that $|h|^{-2}(H_t^{h}(B) -\mathbf E
[H_t^{h}(B)])$ converges to a random variable of the form $c_2
W_{\alpha
}$, with $\alpha$ defined by~(\ref{eqdef-alpha}) and a universal
constant $c_2$.
Nevertheless, $\alpha$ is a divergent quantity in the two-dimensional case
and the convergence of $h^{-3/2}(H_t^{h}(B) -\mathbf E[H_t^{h}(B)])$ is in
fact an empty statement.

In spite of this lack of convergence, the analysis of projections on
chaoses is still a valuable information for two main reasons: (a) It
indicates that a sort of convergence is at least possible for
$H_t^{h}(B)$. (b) We are able to show that the series $\sum_{m\ge1}
\sigma
_{m}^{2}$ is divergent just as in the one-dimensional case, which seems
to indicate that a noncentral limit theorem is to be expected for the
quantity $(H_t^{h}(B) -\mathbf E[H_t^{h}(B)])$.

The methodology we have followed in order to get the results mentioned
above is based on three main ingredients: (a) Stochastic calculus is
obviously important in this Brownian context, and It\^o formulae of
backward type are invoked in order to control terms of the form $\int_0^r e^{\imath\xi(B_r-B_u)} \,du$ (throughout\vspace*{2pt} the paper, we will write
$\imath$ for the complex number $(-1)^{1/2}$). Theorem~\ref{thmlim-hamiltonian-gamma} will also be a consequence of limit
theorems for martingales according to the behavior of their bracket
process. (b) An important contribution comes from stochastic analysis
techniques: our chaos decompositions are obtained through repeated
applications of Stroock's formula and we use representations of
Brownian local times by means of Watanabe distributions. We also derive
central limit theorems on chaoses by analyzing contractions of kernels
for multiple Wiener integrals, as assessed in \cite{NP,peccati-tudor}.
({c}) After application of the high level tools mentioned above, our
results are reduced to rather elementary (though intricate)
computations, for which we resort to Fourier analysis and thorough
analysis of singularities for integrals defined on simplexes. All those
ingredients are detailed in the corresponding sections.

In the remainder of the paper, each section is devoted to the proof of
one of the theorems given above. Specifically, Section~\ref
{secl2-mod-cty-riesz} handles the noncentral limit Theorem~\ref{thmlim-hamiltonian-gamma}
for Riesz type potentials. Section~\ref
{sec1-dim} is concerned with the central limit Theorems~\ref
{thmcvgce-L2-modulus-chaos} for $L^2$ modulus of one-dimensional local
time on chaoses, while Section~\ref{sec2-dim} deals with
generalizations (Theorem~\ref{thmcvgce-L2-modulus-chaos-2d}) to the
two-dimensional case.

\section{$L^2$ modulus of continuity of Brownian Riesz potentials}\label{secl2-mod-cty-riesz}
This section is devoted to the proof of Theorem~\ref{thmlim-hamiltonian-gamma}. We shall first reduce our problem thanks
to an application of Clark--Ocone's formula, and then identify the
limiting process with a combination of Fourier analysis and stochastic
calculus tools.

\subsection{Reduction of the problem}
In order to proceed with our computations, let us first settle some
useful notation:

%
\begin{notation}\label{notdef-f-beta-Q}
The Gaussian heat kernel on $\mathbb R$ is denoted by $p_t(x)$, namely
%
\begin{equation}
\label{eqdef-gauss-kernel} p_t(z) = (2\pi)^{-1/2} \exp \biggl(-
\frac{z^{2}}{2} \biggr), \qquad z\in\mathbb R.
\end{equation}
For $\beta\in(0,1)$, we call
$f_{\beta}\dvtx \mathbb R^{*}\to\mathbb R^{*}$ the function defined by
$f_{\beta
}(x)=|x|^{-\beta}$. For $\beta\in(0,1)$ and $0\leq r\leq t \leq1$, we
also consider the quantity
%
\begin{eqnarray}\label{eqdef-Q-h-beta}
Q^{h,\beta}_{t,r} &=& \int_{r}^{t}
\,ds \int_{0}^{r} \,du \bigl[K_{s-r}^{\beta}(B_r-B_u+h)
\nonumber\\[-8pt]\\[-8pt]
&&\hspace*{61pt}{} + K_{s-r}^{\beta}(B_r-B_u-h) -2
K_{s-r}^{\beta}(B_r-B_u) \bigr],\nonumber
\end{eqnarray}
where $K^\beta_{u}$ stands for the (convolved) Riesz kernel $K^\beta_{u}:=
f_{\beta} \ast p_{u}'$ for all $u\ge0$.
\end{notation}

With these notation in mind, the Hamiltonian $H_t^{h,\gamma}(B)$ can be
expressed as follows.

%
\begin{lemma}\label{lemH-gamma-f-beta}
For $t\in[0,1]$, consider the quantity $H_t^{h,\gamma}(B)$ defined by
(\ref{eqdef-hamiltonian-gamma}). Then
%
\begin{eqnarray}\label{eqH-gamma-f-beta}
H_t^{h,\gamma}(B) &=& c_{\gamma} \int
_{[0,t]^{2}} \bigl[2 f_{\beta}(B_v-B_u)
\nonumber\\[-8pt]\\[-8pt]
&&\hspace*{39pt}{} -f_{\beta}(B_v-B_u+h) - f_{\beta
}(B_v-B_u-h)\bigr] \,du \,dv\nonumber
\end{eqnarray}
with $\beta=2\gamma-1$.
\end{lemma}

\begin{pf}
Start from expression~(\ref{eqdef-hamiltonian-gamma}) and write
$H_t^{h,\gamma}(B)$ as
\begin{eqnarray*}
&& \int_{\mathbb R} \biggl(\int_{[0,t]^{2}}
\bigl[f_{\gamma}(B_v+x+h) - f_{\gamma}(B_v+x)
\bigr]
\\
&&\hspace*{41pt}{}\times \bigl[f_{\gamma}(B_u+x+h) - f_{\gamma}(B_u+x)
\bigr]\,du \,dv \biggr) \,dx.
\end{eqnarray*}
Next expand the product inside the integral, apply Fubini in order to
integrate with respect to the variable $x$ first and apply the identity
$f_{\gamma}\ast f_{\gamma} = c_{\gamma} f_{2\gamma-1}$. Our
claim is easily
deduced from these elementary manipulations.
\end{pf}

We shall now see that Theorem~\ref{thmlim-hamiltonian-gamma} can be
reduced to the following.

\begin{theorem}\label{thmlim-Qh-dB}
For every $\beta\in(1/2,1]$, consider the process $Q^{h,\beta}$ defined
by~(\ref{eqdef-Q-h-beta}). Then the following limit as $h$ tends to
zero holds true in the space $\mathcal{C}([0,1];\mathbb R)$ of real continuous
functions on $[0,1]$:
%
\begin{equation}
\label{main} \frac{\widetilde{Q}{}^{h}}{h^{5/2-\beta}} \xrightarrow{(d)} c_\beta
W_{\alpha} \qquad\mbox{where } \widetilde{Q} {}^{h}_{t}:=
\int_0^t Q^{h,\beta}_{t,r}
\,dB_r.
\end{equation}
Here, $c_\beta$ is a deterministic constant depending only on $\beta
$, and
the process $W_{\alpha}$ has been introduced at equation (\ref
{eqlim-hamiltonian-gamma}).
\end{theorem}

\begin{pf*}{Proof of the equivalence between Theorems~\ref{thmlim-hamiltonian-gamma}~and~\ref{thmlim-Qh-dB}}
Following expression~(\ref{eqH-gamma-f-beta}), set
\[
H_t^{\beta}(B) = - \int_{[0,t]^{2}} \bigl[2
f_{\beta}(B_v-B_u) -f_{\beta}(B_v-B_u+h)
- f_{\beta
}(B_v-B_u-h) \bigr] \,du \,dv.
\]
Then Lemma~\ref{lemH-gamma-f-beta} asserts that Theorem~\ref{thmlim-hamiltonian-gamma} is proved once we can show that the process
$h^{-(7/2-2\gamma)}(H^{\beta}(B) - \mathbf E[H^{\beta}(B)])$
converges in law
to $c_\beta W_\alpha$ for a strictly positive constant $c_\beta$.
It is
obviously easier to express everything in terms of $\beta=2\gamma-1$, so
that we are reduced to show that $h^{-(5/2-\beta)}(H^{\beta}(B) -
\mathbf E
[H^{\beta}(B)])$ converges in law to $c_\beta W_\alpha$. It should also
be observed that if $\gamma\in(3/4,1)$ then $\beta$ lies into $(1/2,1)$.

Now along the same lines as in \cite{HN09}, a direct application of
Clark--Ocone formula enables to express $H_t^{\beta}(B)$ in the
following way:
\[
H_t^{\beta}(B) - \mathbf E \bigl[H_t^{\beta}(B)
\bigr]= \int_0^t Q^{h,\beta}_{t,r}
\,dB_r,
\]
where the process $Q^{h,\beta}$ is defined at Notation~\ref{notdef-f-beta-Q}. This finishes the proof of our equivalence.
\end{pf*}

With this equivalence in hand, the remainder of the section is now
devoted to the proof of Theorem~\ref{thmlim-Qh-dB}. As mentioned in
the \hyperref[sec1]{Introduction}, our strategy to show this result makes use of some
convenient simplifications offered by a Fourier-transform version of
the problem. As a last preliminary step, let us thus write an
alternative expression for the quantity $Q^{h,\beta}_{t,r}$:

%
\begin{lemma}
Let $\beta\in(1/2,1)$ and $0\leq r\leq t \leq1$. Then
%
\begin{equation}
\label{four-q} \qquad Q^{h,\beta}_{t,r}= \frac{4\imath}{\pi}\int
_{\mathbb R} \biggl[ \bigl(1- e^{-(1/2)\xi
^2(t-r)} \bigr)\psi ( h\xi )
\frac{\xi}{|\xi
|^{3-\beta}} \int_0^r e^{\imath\xi(B_r-B_u)} \,du
\biggr] \,d\xi,
\end{equation}
where $\psi\dvtx \mathbb R\to\mathbb R$ stands for the function defined by
$\psi(\xi):=
\sin^2(\xi/2)$.
\end{lemma}

\begin{pf}
It is well known that for all $x\in\mathbb R^{*}$ we have
\[
K^\beta_t(x)=-\frac{\imath}{2\pi} \int_{\mathbb R}
e^{\imath\xi x} \frac{\xi
}{|\xi|^{1-\beta}} e^{-(t \xi^2)/2} \,d\xi.
\]
Plugging this identity into~(\ref{eqdef-Q-h-beta}) and applying
Fubini's theorem, we get
\[
-\frac{\imath}{2\pi} \int_{\mathbb R} \biggl[ \biggl(\int
_{r}^{t} e^{-{((s-r) \xi^2)}/{2}} \,ds \biggr) \int
_{0}^{r} \frac{\xi}{|\xi|^{1-\beta}} e^{\imath\xi(B_r-B_u)}
\bigl(e^{\imath\xi h} + e^{-\imath\xi h} -2 \bigr) \,du \biggr] \,d\xi
\]
from which identity~(\ref{four-q}) is easily deduced.
\end{pf}

We now start by identifying the main contribution in the quantity $\int_0^t Q^{h,\beta}_{t,r} \,dB_r$ appearing in~(\ref{main}) by means of
our Fourier representation~(\ref{four-q}).

\subsection{Elimination of some negligible terms}

The\vspace*{1pt} first term which might yield a negligible contribution in
$\widetilde{Q}{}^{h}$
is given\vspace*{1pt} by the small exponential term $e^{-((t-r) \xi^2)/{2}}$ in
expression~(\ref{four-q}). We thus set $Q^{h,\beta}_{t,r}=Q^{h,\beta,1}_{r}-A^{h}_{t,r}$, with
%
\begin{eqnarray}
Q^{h,\beta,1}_{r} &=& \frac{4\imath}{\pi}\int_{\mathbb R}
\biggl[\psi ( h\xi ) \frac{\xi}{|\xi|^{3-\beta}} \int_0^r
e^{\imath\xi(B_r-B_u)} \,du \biggr] \,d\xi,\label{eqdef-Q-h-1}
\\
A^{h}_{t,r} &=& \frac{4\imath}{\pi}\int_{\mathbb R}
\biggl[e^{-(1/2)\xi^2(t-r)} \psi ( h\xi ) \frac{\xi}{|\xi|^{3-\beta}} \int
_0^r e^{\imath\xi(B_r-B_u)} \,du \biggr] \,d\xi.
\label{eqdef-A-h}
\end{eqnarray}
Then the following proposition identifies a first vanishing term.

%
\begin{proposition}\label{propa}
Let $A^{h}$ be the process defined by~(\ref{eqdef-A-h}), and for
$t\in
[0,1]$ set
\[
\widetilde{A} {}^h_t:=\frac{1}{h^{5/2-\beta}} \int
_0^t A_{t,r}^h
\,dB_r.
\]
Then we have:
\begin{longlist}[(iii)]
\item[(i)] For every fixed $t\in[0,1]$, $\widetilde{A}{}^h_t \to0$ in
$L^2(\Omega)$ as $h$ tends to zero.

\item[(ii)] There exist $p\geq1$ and $\alpha>0$ such that for all
$0\leq s <t \leq1$ and every $h\in(0,1)$,
\[
\mathbf E \bigl[ \bigl| \widetilde{A} {}^h_t- \widetilde{A}
{}^h_s \bigr|^{2p} \bigr] \leq c_p
h^{2p(\beta-(1/2))} \llvert t-s \rrvert ^{1+\alpha}
\]
for some constant $c_p$ depending only on $p$.

\item[(iii)]
As a consequence, we have $\widetilde{A}{}^h \xrightarrow{(d)} 0$ in
$\mathcal
{C}([0,1];\mathbb R)$ as $h$ tends to zero.
\end{longlist}
\end{proposition}

\begin{pf} Let us prove the three items separately:
\begin{longlist}[(ii)]
\item[(i)] Consider a given $t\in[0,1]$. One has
\begin{eqnarray*}
&& \mathbf E \biggl[ \biggl(\int_0^t
A_{t,r}^h \,dB_r \biggr)^2 \biggr]
\\
&&\qquad = 2 \int_0^t dr \int_{\mathbb R}
d\xi\int_{\mathbb R} d\eta\int_0^r
dv \int_0^v du
\frac{\xi e^{-(1/2)\xi^2(t-r)} \psi(h\xi)}{|\xi|^{3-\beta}}
\\
&&\hspace*{178pt}{}\times \frac{\eta e^{-(1/2)\eta^2(t-r)} \psi(h\eta)}{|\eta
|^{3-\beta}}
\\
&&\hspace*{178pt}{}\times  \mathbf E \bigl[ e^{\imath(\xi+\eta)(B_r-B_v)+\imath\eta
(B_v-B_u)} \bigr].
\end{eqnarray*}
Furthermore, for $u<v<r<t$ we have
\[
0 \le\mathbf E \bigl[e^{\imath(\xi+\eta)(B_r-B_v)+\imath\eta
(B_v-B_u)} \bigr] = e^{-((\xi+\eta)^2/2)(r-v)} e^{-(\eta^2/2)(v-u)}
\le e^{-(\eta^2/2)(v-u)}.
\]
Now integrate this inequality in $u$ and invoke the fact that $\psi
(z)\le c z^2$ in order to get
%
\begin{equation}
\label{majo} \mathbf E \biggl[ \biggl(\int_0^t
A_{t,r}^h \,dB_r \biggr)^2 \biggr]
\le c h^4 \int_0^t dr \int
_0^r dv \int_{\mathbb R} d\xi
\int_{\mathbb R} d\eta\, \ell_{r,v}^{t}(\xi,
\eta),
\end{equation}
where
\[
\ell_{r,v}^{t}(\xi,\eta) \equiv e^{-(1/2)\xi^2(t-r)}e^{-(1/2)\eta^2(t-r)}
\frac{|\xi|^\beta
}{|\eta
|^{2-\beta}} \bigl\{1-e^{-(\eta^2/2)v} \bigr\}.
\]
To see that the integral in the right-hand side of~(\ref{majo}) is
indeed finite, observe first that $\int_{\mathbb R} e^{-(a/2)\xi^2}|\xi
|^\beta \,d\xi= c_{\beta} a^{-(1+\beta)/2}$ for any $a>0$
and $\beta\in
(0,1)$. Thus,
\begin{eqnarray*}
\int_{\mathbb R} d\xi\int_{-1}^1
d\eta \ell_{r,v}^{t}(\xi,\eta) &\leq& c \int
_{\mathbb R} e^{-(1/2)\xi^2(t-r)} |\xi|^\beta \,d\xi \int
_{-1}^1 |\eta|^\beta \,d\eta
\\
&\leq& c \int_{\mathbb R} d\xi e^{-(1/2)\xi^2(t-r)}|\xi|^\beta
\\
&\leq&\frac{c}{|t-r|^{(1+\beta)/{2}}}.
\end{eqnarray*}
In the same way, since $\beta\in(0,1)$ we also have
\begin{eqnarray*}
\int_{\mathbb R} d\xi\int_{|\eta|\geq1} d\eta \ell
_{r,v}^{t}(\xi,\eta) &\leq& c \int_{|\eta|\geq1} |
\eta|^{-(2-\beta)} \,d\eta \int_{\mathbb R} e^{-(1/2)\xi^2(t-r)}|
\xi|^\beta \,d\xi
\\
&\leq&\frac{c}{|t-r|^{(1+\beta)/{2}}}.
\end{eqnarray*}
Plugging these estimates into~(\ref{majo}) and taking into account the
fact that $\beta\in(0,1)$, we end up with
\[
\mathbf E \biggl[ \biggl(\int_0^t
A_{t,r}^h \,dB_r \biggr)^2 \biggr]
\le c_{t,\beta} h^{4} \int_{0}^{t}
\frac{dr}{|t-r|^{({1+\beta})/{2}}} \le c_{t,\beta} h^{4},
\]
which yields our first claim (i).

\item[(ii)] In order to bound the increment $\widetilde
{A}{}^h_t-\widetilde{A}{}^h_s$, set
\[
k_{h,t}(\xi):=e^{-(1/2) \xi^2 t}
\psi(h\xi) \frac{\xi
}{|\xi|^{3-\beta}}.
\]
Then it is readily checked that
\begin{eqnarray}
\label{eqdcp-tilde-A-h} \widetilde{A} {}^h_t-\widetilde{A}
{}^h_s &=& \frac{1}{h^{5/2-\beta}}\int_s^t
dB_r \int_{\mathbb R} d\xi k_{h,t-r}(\xi) \int
_0^r e^{\imath\xi(B_r-B_u)} \,du\nonumber\hspace*{-15pt}
\\
&&{} +\frac{1}{h^{5/2-\beta}}\int_0^s
dB_r \int_{\mathbb R} d\xi \bigl[ k_{h,t-r}(
\xi)-k_{h,s-r}(\xi) \bigr] \int_0^r
e^{\imath\xi
(B_r-B_u)} \,du\hspace*{-15pt}
\\
&:=& \widetilde{A} {}^{h,1}_{s,t}+
\widetilde{A} {}^{h,2}_{s,t}.\nonumber\hspace*{-15pt}
\end{eqnarray}
Consider first $\widetilde{A}{}^{h,1}_{s,t}$ and write
\[
\widetilde{A} {}^{h,1}_{s,t} = \frac{1}{h^{5/2-\beta}}\int
_s^t H_{t,r} \,dB_r \qquad
\mbox{with } H_{t,r}:=\int_{\mathbb R} d\xi
k_{h,t-r}(\xi) \int_0^r
e^{\imath
\xi
(B_r-B_u)} \,du.
\]
By using successively Burkholder--Davies--Gundy and Cauchy--Schwarz
inequalities, we get
%
\begin{equation}
\label{esti-proof} \mathbf E \bigl[ \bigl| \widetilde{A} {}^{h,1}_{s,t}\bigr|^{2p}
\bigr] \leq \frac{c_p}{h^{(5-2\beta)p}} \biggl( \int_s^t
\mathbf E \bigl[|H_{t,r} |^{2p} \bigr]^{1/p} \,dr
\biggr)^p,
\end{equation}
%
with
\begin{eqnarray*}
\mathbf E \bigl[|H_{t,r} |^{2p} \bigr] &=&c_p
\int_{\mathbb R^{2p}} d\xi_1 \cdots \,d\xi_{2p} \int
_{0<u_1<\cdots<u_{2p}
<r} du_1 \cdots \,du_{2p}
\\
&&{}\times \prod_{j=1}^{2p} k_{h,t-r}(
\xi_j) \mathbf E \bigl[e^{\imath\xi
_j(B_r-B_{u_j})} \bigr],
\end{eqnarray*}
which can also be expressed as
\begin{eqnarray*}
\mathbf E \bigl[|H_{t,r} |^{2p} \bigr] &=& c_p
\int_{\mathbb R^{2p}} d\xi_1 \cdots \,d\xi_{2p} \int
_{0<u_1<\cdots<u_{2p}
<r} du_1 \cdots \,du_{2p}
\\
&&{}\times \prod_{j=1}^{2p} k_{h,t-r}(
\xi_j) e^{-(1/2) \xi_1^2 (u_2-u_1)}
\\
&&\hspace*{26pt}{}\times  e^{-(1/2) (\xi_1+\xi_2)^2 (u_3-u_2)} \cdots e^{-(1/2) (\xi_1+\cdots+\xi_{2p})^2 (r-u_{2p})}.
\end{eqnarray*}
We can then rely on the uniform estimate
\[
\bigl| k_{h,t-r}(\xi_i) \bigr|\leq c h^2 e^{-(1/2) \xi_i^2(t-r)}
|\xi _i|^\beta \leq c \frac{h^2}{\llvert t-r \rrvert ^{\beta/2}}
\]
and the fact that
\begin{eqnarray*}
&& \int_{\mathbb R} d\xi_1 e^{-(1/2) \xi_1^2
(u_2-u_1)}\int
_{\mathbb R} d\xi_2 e^{-(1/2) (\xi_1+\xi_2)^2 (u_3-u_2)}\cdots
\\
&&\quad{} \times\int
_{\mathbb R} \,d\xi _{2p} e^{-(1/2) (\xi_1+\cdots+\xi_{2p})^2 (r-u_{2p})}
\\
&&\qquad = \int_{\mathbb R} d\xi_1 e^{-(1/2) \xi_1^2 (u_2-u_1)}\int
_{\mathbb R} d\xi _2 e^{-(1/2) \xi_2^2 (u_3-u_2)} \cdots\int
_{\mathbb R} d\xi _{2p} e^{-(1/2) \xi_{2p}^2 (r-u_{2p})}
\\
&&\qquad = c_p (u_2-u_1)^{-1/2}
(u_3-u_2)^{-1/2} \cdots(r-u_{2p})^{-1/2}
\end{eqnarray*}
in order to get
\[
\mathbf E \bigl[|H_{t,r} |^{2p} \bigr] \le
\frac{c_p h^{4p}
r^{p}}{|t-r|^{\beta p}}.
\]
Plugging this estimate into~(\ref{esti-proof}), we end up with
%
\begin{equation}
\label{eqbound-tilde-A-1} E \bigl[ \bigl| \widetilde{A} {}^{h,1}_{s,t}\bigr|^{2p}
\bigr]\leq c_p h^{2p(\beta-1/2)} \llvert t-s \rrvert ^{(1-\beta)p}.
\end{equation}
The bound for $\widetilde{A}{}^{h,2}_{s,t}$ can be derived from a similar
procedure. Observe, for instance, that
\begin{eqnarray*}
\bigl| k_{h,t-r}(\xi)-k_{h,s-r}(\xi) \bigr| &\leq& h^2 \bigl|
e^{-(1/2) \xi
^2(t-r)}-e^{-(1/2) \xi^2 (s-r)} \bigr| |\xi|^\beta
\\
&\leq& c h^2
\llvert t-s \rrvert ^{\varepsilon} \llvert s-r \rrvert ^{(1-\varepsilon)/2}
\end{eqnarray*}
and invoking this bound for $\varepsilon:=(1-\beta)/3$ one obtains that
inequality~(\ref{eqbound-tilde-A-1}) also holds true for $\widetilde
{A}{}^{h,2}_{s,t}$. Going back to~(\ref{eqdcp-tilde-A-h}), we see that
the bounds on $\widetilde{A}{}^{h,1}_{s,t}$ and $\widetilde{A}{}^{h,2}_{s,t}$
easily yield our claim (ii). Assertion (iii) is now a standard
consequence of (i)~and~(ii).\quad\qed
\end{longlist}\noqed
\end{pf}

Let\vspace*{1pt} us go back to expression~(\ref{four-q}), as well as the
decomposition~(\ref{eqdef-Q-h-1}) and~(\ref{eqdef-A-h}) for
$Q^{h,\beta
}$. Proposition~\ref{propa} allows to reduce our study to an analysis
of $\widetilde{Q}{}^{h,\beta,1}$ defined by $\widetilde{Q}{}^{h,\beta,1}_{t}=
h^{-(5/2-\beta)}\int_{0}^{t} Q^{h,\beta,1}_{r} \,dB_r$, where $Q^{h,\beta,1}$ is given by
(\ref{eqdef-Q-h-1}). In order to identify another negligible term
within $\widetilde{Q}{}^{h,\beta,1}$, let us resort to It\^o's formula
applied to the
(backward) Brownian motion $\widehat{B}{}^r=\{B_r-B_u; 0\le u \le r\}
$ and
$f(x):=e^{\imath\xi x}$. This gives
%
\begin{equation}
\label{eqito-exp-iBt} \int_0^r e^{\imath\xi(B_r-B_u)} \,du
= - \frac{2 (e^{\imath\xi B_r} -1  )}{\xi^2} + \frac
{2\imath}{\xi} \int_0^r
e^{\imath(B_r-B_u)} \,d\widehat{B} {}^r_u
\end{equation}
and plugging this identity into~(\ref{eqdef-Q-h-1}) we get
$Q^{h,\beta,1}_{r}=D^{h}_{r}-Q^{h,\beta,2}_{r}$, with
%
\begin{eqnarray}
D^{h}_{r} &=& \frac{8\imath}{\pi}\int_{\mathbb R}
\biggl[\frac{\xi
\psi( h\xi
)}{|\xi|^{5-\beta}} \bigl(e^{\imath\xi B_r} -1 \bigr) \biggr] \,d
\xi,\label{eqdef-D-h}
\\
Q^{h,\beta,2}_{r} &=& \frac{8}{\pi}\int_{\mathbb R}
\biggl[\frac
{\psi ( h\xi
)}{|\xi|^{3-\beta}} \int_0^r
e^{\imath\xi(B_r-B_u)} \,d\widehat{B} {}^r_u \biggr] \,d\xi.
\label{eqdef-Q-h-2}
\end{eqnarray}
We now prove the following proposition.

%
\begin{proposition}\label{propD-vanish}
Let $D^{h}$ be the process defined by~(\ref{eqdef-D-h}), and for
$t\in
[0,1]$ set
\[
\widetilde{D} {}^h_t:=\frac{1}{h^{5/2-\beta}} \int
_0^t D_{r}^h
\,dB_r.
\]
Then the conclusions of Proposition~\ref{propa} hold true for
$\widetilde{D}{}^h$.
\end{proposition}

\begin{pf}
The proof goes along the same lines as for Proposition~\ref{propa},
and is left to the reader for the sake of conciseness. Let us just
highlight the following decomposition:
\[
\mathbf E \bigl[ \bigl(\widetilde{D} {}^h_t
\bigr)^{2} \bigr] \leq c h^{2\beta-1} \int_0^t
\mathbf E^{2} \bigl[ B_r^2 \bigr] \,dr \biggl(
\int_{-1}^1 \frac{d\xi}{|\xi|^{1-\beta}}
\biggr)^2+c h^{2\beta
-1} \biggl(\int_{|\xi|\geq1}
\frac{d\xi}{|\xi|^{2-\beta}} \biggr)^2,
\]
which allows us to conclude that $\lim_{h\to0}\mathbf E[( \widetilde{D}{}^h_t
)^{2}]=0$ since $1/2<\beta<1$.
\end{pf}

%
\begin{remark}
With Propositions~\ref{propa}~and~\ref{propD-vanish} in
hand, Theorem~\ref{thmlim-Qh-dB} now boils down to the following property:
%
\begin{equation}
\label{mart} \frac{M^h}{h^{5/2-\beta}} \xrightarrow{(d)} c_\beta
W_{\alpha
}\qquad\mbox{in } \mathcal{C} \bigl([0,1];\mathbb R \bigr)
\mbox{ with } M^{h}_{t}:= \int_{0}^{t}
Q^{h,\beta,2}_{r} \,dB_{r},
\end{equation}
where $Q^{h,\beta,2}$ is the process defined by~(\ref{eqdef-Q-h-2}). It
should be observed that $M^{h}$ is now a Brownian martingale, for which
specific limit theorems are available.
\end{remark}

\subsection{Study of the martingale term}
Similar to the argument used in \cite{HN09,HN10,Ro11}, our strategy
toward~(\ref{mart}) is now based on the martingale convergence
criterion summed up in \cite{hayashi-mykland}, Theorem~A.1. Using the
latter result, the proof of~(\ref{mart}) reduces to showing that, as
$h\to0$, we have simultaneously
%
\begin{equation}
\label{mart-equiv} \frac{\langle M^h,B \rangle_t}{h^{5/2-\beta}} \to0\quad\mbox{and}\quad \frac{\langle M^h \rangle_t}{h^{5-2\beta}}
\to c_\beta \alpha_t
\end{equation}
in $L^2(\Omega)$ for every fixed $t\in[0,1]$, with $\alpha_t$
defined by
(\ref{eqdef-alpha}).

To this aim, let us start by recasting $M^h$ in a suitable way. Indeed,
thanks to a stochastic Fubini theorem we have
%
\begin{equation}
\label{eqrep-Q2-normalized} \frac{Q^{h,\beta,2}_{r}}{h^{5/2-\beta}}= \int_0^r
g_h(B_r-B_u) \,d \widehat{B}
{}^r_u,
\end{equation}
where
%
\begin{equation}
\label{eqdef-gh-fh} g_h=g_h^\beta:=
\mathcal{F}(f_h) \qquad\mbox{with } f_h(\xi
)=f_h^\beta(\xi):=\frac{1}{h^{5/2-\beta}} \frac{\psi(h\xi)}{|\xi|^{3-\beta}}.
\end{equation}
In the course of the reasoning, we shall appeal to the following key
properties of~$g_h$:

%
\begin{lemma}\label{lemg-h}
It holds that:
\begin{longlist}[(iii)]
\item[(i)] For some $c_{\beta}$ independent of $h$, we have $\int_{\mathbb R}
g_h(x)^2 \,dx=c_{\beta} >0$.

\item[(ii)] Recalling that $p_t$ stands for the Gaussian heat kernel
defined by~(\ref{eqdef-gauss-kernel}), we have for every $t\in(0,1]$:
%
\begin{equation}
\label{regu} \int_{\mathbb R} g_h(x)p_t(x)
\,dx \leq\frac{c h^{\beta
-1/2}}{t^{\beta/2}}.
\end{equation}

\item[(iii)]
The function $g_h$ can also be written as
%
\begin{equation}
\label{repres} g_h(x)=\frac{c}{h^{5/2-\beta}} \int_{x-h}^{x+h}
\frac
{(h-|x-y|)}{|y|^\beta} \,dy.
\end{equation}
In particular, $g_h(-x)=g_h(x)$ and $g_h(x)\geq0$ for all $x\in
\mathbb R$.

\item[(iv)] For every $\varepsilon>0$ such that $\beta
>1/2+\varepsilon$, every $h\leq
1/4$ and every $|x|\geq\sqrt{h}$,
%
\begin{equation}
g_h^\beta(x)\leq c h^{\varepsilon/2} g_h^{\beta-\varepsilon}(x).
\end{equation}
\end{longlist}
\end{lemma}

\begin{pf}
By Fourier isometry,
\[
\| g_h\|_{L^2}^2=\|f_h
\|_{L^2}^2 =\frac{1}{h^{5-2\beta}} \int_{\mathbb R}
\frac{\psi^2(h\xi)}{|\xi
|^{6-2\beta}} \,d\xi =\int_{\mathbb R} \frac{\psi^2(\xi)}{|\xi|^{6-2\beta}} \,d\xi,
\]
which gives (i). In order to prove (ii) use Fourier isometry again,
which according to~(\ref{eqdef-gh-fh}) yields
\begin{eqnarray*}
\int_{\mathbb R} g_h(x)p_t(x) \,dx &=&
\frac{c}{h^{5/2-\beta}} \int_{\mathbb R} \frac{\psi(h\xi)}{|\xi
|^{3-\beta}}
e^{-(t\xi^2)/2} \,d\xi
\\
&\le& c h^{\beta-1/2} \int_{\mathbb R}
\frac{e^{-(t\xi^2)/2}}{|\xi|^{1-\beta
}} \,d\xi \le\frac{c h^{\beta-1/2}}{t^{\beta/2}}.
\end{eqnarray*}

For (iii), observe that
\[
f_h(\xi)=h^{1/2} \varphi(h\xi)\qquad\mbox{with }
\varphi(u)=\frac{\operatorname{sinc}^2(u)}{|u|^{1-\beta}},
\]
where the $\operatorname{sinc}$ function refers to $\operatorname
{sinc}(x)=\frac{\sin
(x)}{x}$. Thus, using the fact\break  $\mathcal{F}(\operatorname
{sinc}^2(\cdot))(\xi
)=\mathbf{1}_{[-1,1]}(\xi)(1-|\xi|)$, we get
\begin{eqnarray*}
g_h(\xi) &=& \mathcal{F}(f_h) (\xi) = \frac{1}{h^{1/2}}
\mathcal{F}(\varphi) \biggl( \frac{\xi}{h} \biggr) = \frac{1}{h^{1/2}}
\bigl[ \mathcal{F} \bigl(|\cdot|^{-1+\beta
} \bigr)\ast \mathcal{F} \bigl(
\operatorname{sinc}^2(\cdot) \bigr) \bigr] \biggl( \frac{\xi}{h}
\biggr)
\\
&=&\frac{c}{h^{1/2}} \int_{(\xi/h)-1}^{(\xi/h)+1}
\frac
{dy}{|y|^\beta} \biggl( 1- \biggl| \frac{\xi}{h}-y \biggr| \biggr),
\end{eqnarray*}
which clearly leads to~(\ref{repres}).

Now we can use~(\ref{repres}) in order to prove (iv): for $x>\sqrt {h}$, write
\[
g_h^\beta(x)= c h^\varepsilon\frac{1}{h^{5/2-(\beta-\varepsilon
)}} \int
_{x-h}^{x+h} \frac{h-|x-y|}{|y|^\varepsilon|y|^{\beta-\varepsilon}} \,dy
\\
\leq\frac{c h^\varepsilon g_h^{\beta-\varepsilon
}(x)}{|x-h|^\varepsilon} \leq c h^{\varepsilon
/2} g_h^{\beta-\varepsilon}(x),
\]
since $|x-h| \geq\frac{1}2 \sqrt{h}$. By symmetry of $g_h$, this
completes our proof.
\end{pf}

Let us develop now the strategy for the convergence of the martingale
term, which has been summarized in~(\ref{mart-equiv}). We shall prove
the first claim of~(\ref{mart-equiv}), namely the following.

%
\begin{proposition}\label{propbracket-M-B-vanish}
For all $t\in[0,1]$, the martingale term $M^{h}$ satisfies
\[
\frac{\mathbf E [\langle M^h,B\rangle_{t}^{2} ]}{h^{5-2\beta
}} \le c_{t} h^{\beta-1/2},
\]
where $c_t$ is a uniformly bounded function of $t\in[0,1]$.
\end{proposition}

\begin{pf}
According to~(\ref{mart}) and~(\ref{eqrep-Q2-normalized}), one has
\[
\frac{\langle M^h,B\rangle_t}{h^{5/2-\beta}} = \frac{\int_{0}^{t} Q_{r}^{\beta,h,2} \,dr}{h^{5/2-\beta}} = \int_0^t
dr \int_0^r d\widehat{B} {}^r_u
g_h(B_r-B_u).
\]
Hence,
\[
\frac{\mathbf E [\langle M^h,B\rangle_t^{2} ]}{h^{5-2\beta}} =2\int_0^t
dr_1 \int_0^{r_1} dr_2
\int_0^{r_2} du \mathbf E \bigl[
g_h(B_{r_1}-B_u)g_h(B_{r_2}-B_u)
\bigr]
\]
and furthermore,
\begin{eqnarray*}
&& \mathbf E \bigl[ g_h(B_{r_1}-B_u)g_h(B_{r_2}-B_u)
\bigr]
\\
&&\qquad  = \mathbf E \bigl[ g_h \ast p_{r_1-r_2}(B_{r_2}-B_u)
g_h(B_{r_2}-B_u) \bigr]
\\
&&\qquad = \int_{\mathbb R} d\xi [g_h \ast p_{r_1-r_2} ]
(\xi) g_h(\xi) p_{r_2-u}(\xi) \leq c \|g_h \ast
p_{r_1-r_2}\|_\infty\frac{h^{\beta-1/2}}{\sqrt{r_2-u}},
\end{eqnarray*}
thanks to~(\ref{regu}). In addition, $\|g_h \ast p_{r_1-r_2}\|_\infty
\leq\|g_h\|_{L^2} \|p_{r_1-r_2}\|_{L^2} \leq c |r_1-r_2|^{-1/4}$, and thus
\[
\frac{\mathbf E [\langle M^h,B\rangle_t^{2} ]}{h^{5-2\beta}} \le c h^{\beta-1/2} \int_0^t
dr_1 \int_0^{r_1} dr_2
|r_1-r_2|^{-1/4} \int_0^{r_2}
du |r_2-u|^{-1/2}
\]
from which our claim is easily deduced.
\end{pf}

Before we proceed with the proof of~(\ref{mart-equiv}), let us label a
technical lemma on Brownian local times.

%
\begin{lemma}\label{lemloc-time-BM}
Let $\{L_t(a); t\in[0,1], a\in\mathbb R\}$ be the local time
process of
Brownian motion on the interval $[0,1]$. Then there exist $\varepsilon
>0$ and a
strictly positive constant $c$ such that
\[
\sup_{x\in\mathbb R,t\in[0,1]}\mathbf E \bigl[\bigl|L_t(x+B_t)\bigr|^{2}
\bigr] \le c
\]
and
\[
\sup_{t\in[0,1]} \mathbf E \Bigl[
\sup_{|x-y|<h^{1/2}} \bigl\llvert L_t(x)-L_t(y)
\bigr\rrvert ^{2} \Bigr] \le c h^{\varepsilon}.
\]
\end{lemma}

\begin{pf}
By applying Tanaka's formula to the backward Brownian motion $\widehat B$,
we get, for all $x\in\mathbb R$,
%
\begin{equation}
\label{eqtanaka-backward} \bigl|L_t(x+B_t)\bigr| \leq2 |B_t|+2
\biggl\llvert \int_0^t \mathbf{1}_{\{\widehat
{B}{}^t_s<-x\}}
\,d\widehat B^t_s \biggr\rrvert
\end{equation}
and the first assertion immediately follows. The second assertion of
our lemma can be derived from \cite{barlow-yor}, item (ii).
\end{pf}

We are now ready to prove the second part of assertion (\ref
{mart-equiv}), that is, the following proposition.

%
\begin{proposition}
Let $t$ be an arbitrary time in $[0,1]$. Then we have:
%
\begin{equation}
\label{eqL2-cvgce-bracket-Mh} L^{2}(\Omega)-\lim_{h\to0}
\frac{\langle M^h \rangle_t}{h^{5-2\beta}} = c_\beta \alpha_t,
\end{equation}
where $\alpha$ is the self-intersection local time defined by (\ref
{eqdef-alpha}).
\end{proposition}

\begin{pf}
Let us start by applying again the backward It\^o formula (\ref
{eqito-exp-iBt}) in order to get the decomposition
\[
\bigl\langle M^h \bigr\rangle_t=\int
_0^t \,dr \biggl( \int_0^r
\,d\widehat{B} {}^r_u g_h(B_r-B_u)
\biggr)^2:= N^{h,1}_{t} + N^{h,2}_{t}
\]
with
\begin{eqnarray*}
N^{h,1}_{t} &=& \int_0^t
\,dr \int_0^r \,du \bigl[g_h(B_r-B_u)
\bigr]^2,
\\
N^{h,2}_{t} &=& 2\int_0^t
\,dr \int_0^r \,d\widehat{B} {}^r_u
\biggl( g_h(B_r-B_u) \int
_u^r \,d\widehat{B} {}^r_s
g_h(B_r-B_s) \biggr).
\end{eqnarray*}
We shall now divide our proof in two steps.
\begin{longlist}[\textit{Step} 1.]
\item[\textit{Step} 1.]
\textit{$N^{h,2}$ vanishes as $h\to0$.} Specifically, we shall
prove that $L^{2}(\Omega)-\lim_{h\to0} N^{h,2}_{t}=0$. Indeed, it is
readily checked that
%
\begin{eqnarray}
\label{proof-n-h-2}
&& \mathbf E \biggl[ \biggl(\int_0^t
dr \int_0^r d\widehat B^r_u
\biggl( g_h(B_r-B_u) \int
_u^r d\widehat{B} {}^r_s
g_h(B_r-B_s) \biggr) \biggr)^2
\biggr]
\nonumber
\\
&&\qquad = 2\int_0^t ds \int_s^t
du \int_u^t dr_1 \int
_u^{r_1} dr_2 \mathbf E
\bigl[g_h(B_{r_1}-B_s)g_h(B_{r_1}-B_u)
\\
&&\hspace*{171pt}{}\times
g_h(B_{r_2}-B_s)g_h(B_{r_2}-B_u)
\bigr].\nonumber
\end{eqnarray}
Furthermore, using the fact that $g_h$ is positive (Lemma~\ref{lemg-h}(iii)), we have, for fixed $0<s<u<r_2<r_1<t$,
\begin{eqnarray*}
&& \mathbf E \bigl[ g_h(B_{r_1}-B_s)g_h(B_{r_1}-B_u)|
\mathcal {F}_{r_2} \bigr]
\\
&&\qquad =\int_{\mathbb R}
g_h(x+B_{r_2}-B_s)g_h(x+B_{r_2}-B_u)p_{r_1-r_2}(x)
\,dx
\\
&&\qquad \leq \|p_{r_1-r_2}\|_\infty\|g_h
\|_{L^2}^2 \leq\frac{c}{\sqrt {r_1-r_2}},
\end{eqnarray*}
where we have used Lemma~\ref{lemg-h}(i), and
\begin{eqnarray*}
&& \mathbf E \bigl[ g_h(B_{r_2}-B_s)g_h(B_{r_2}-B_u)
\bigr]
\\
&&\qquad = \mathbf E \bigl[ g_h(B_{r_2}-B_u)
(g \ast p_{u-s}) (B_{r_2}-B_u) \bigr]
\\
&&\qquad \leq \| g_h \ast p_{u-s}\|_\infty\int
_{\mathbb R} g_h(x)p_{r_2-u}(x) \,dx
\\
&&\qquad \leq c \|g_h\|_{L^2} \|p_{u-s}
\|_{L^2} \frac{h^{\beta
-1/2}}{\sqrt{r_2-u}}
\end{eqnarray*}
with the help of Lemma~\ref{lemg-h}(ii). Going back to (\ref
{proof-n-h-2}), the result easily follows.
\end{longlist}

\begin{longlist}[\textit{Step} 2.]
\item[\textit{Step} 2.] \textit{Limit of $N^{h,1}$.}
We will show the following property:
%
\begin{equation}
\label{eqcvgce-L2-Nh1} \int_0^t dr \int
_0^r du \bigl[g_h(B_r-B_u)
\bigr]^2\xxrightarrow {h\to0} c_{\beta} \int
_0^t dr L_r(B_r)\qquad
\mbox{in } L^2(\Omega),
\end{equation}
where $c_{\beta}$ is the constant defined at Lemma~\ref{lemg-h}.
To this aim, observe that according to the occupation density formula
we have
\[
\Delta_{h}:=\int_0^t dr \int
_0^r du \bigl[g_h(B_r-B_u)
\bigr]^2 -c_{\beta} \int_0^t dr
L_r(B_r) = \int_0^t
\biggl(\int_{\mathbb R} Z_r(x) \,dx \biggr) \,dr,
\]
where $Z$ is the process defined by
\[
Z_r(x) = g_h(B_r-x)^2 \bigl[
L_r(x)-L_r(B_r) \bigr].
\]
Next, we decompose $\Delta_{h}$ as $\Delta_{h}^{1}+\Delta_{h}^{2}$, where
\[
\Delta_{h}^{1} = \int_0^t
\biggl(\int_{|x-B_r| <h^{1/2}} Z_r(x) \,dx \biggr)\,dr
\]
and
\[
\Delta_{h}^{2} = \int_0^t
\biggl(\int_{|x-B_r| \ge h^{1/2}} Z_r(x) \,dx \biggr)\,dr.
\]
We now estimate those two terms separately.

The term $\Delta_{h}^{1}$ can be bounded as follows: owing to Lemma
\ref{lemg-h}(i), we have
\[
\Delta_{h}^{1} \le c \int_0^t
\sup_{|x-y|<h^{1/2}} \bigl|L_r(x)-L_r(y)\bigr| \,dr.
\]
Owing to Lemma~\ref{lemloc-time-BM}, we thus get
\[
\mathbf E \bigl[\bigl| \Delta_{h}^{1} \bigr|^{2} \bigr]
\le c \sup_{t\in
[0,1]} \mathbf E \Bigl[ \sup_{|x-y|<h^{1/2}}
\bigl| L_t(x)-L_t(y) \bigr|^{2} \Bigr] \le c
h^{\varepsilon}
\]
for some constant $\varepsilon\in(0,1)$.

As far as $\Delta_{h}^{2}$ is concerned, invoke Lemma~\ref{lemg-h}(iv) in order to conclude that for any $\varepsilon>0$ such that
$\beta>
\frac{1}2+\varepsilon$ and every $h\leq1/4$, we have
\begin{eqnarray*}
\mathbf E \bigl[\bigl| \Delta_{h}^{2} \bigr|^{2} \bigr]
&\le& c h^\varepsilon \int_{0}^{t} \mathbf E
\biggl[ \biggl(\int_{|x-B_r| \ge h^{1/2}} \bigl| g_h^{\beta-\varepsilon
}(x-B_r)\bigr|^2
\bigl\llvert L_r(x)-L_r(B_r) \bigr\rrvert \,dx
\biggr)^{2} \biggr] \,dr
\\
&\le& c h^\varepsilon\int_{0}^{t} \mathbf E
\biggl[\int_{\mathbb R} \bigl| g_h^{\beta-\varepsilon
}(x-B_r)\bigr|^2
\bigl\llvert L_r(x)-L_r(B_r) \bigr\rrvert
^{2} \,dx \biggr] \,dr \le c h^\varepsilon,
\end{eqnarray*}
where we have appealed to Lemma~\ref{lemloc-time-BM} for the last inequality.
\end{longlist}
\begin{longlist}[\textit{Step} 3.]
\item[\textit{Step} 3.] \textit{Conclusion.}
Putting together the bounds on $\Delta_{h}^{1}$ and $\Delta_{h}^{2}$,
we have proved our assertion~(\ref{eqcvgce-L2-Nh1}), which easily yields
\[
L^{2}(\Omega)-\lim_{h\to0} \frac{\langle M^h \rangle_t}{h^{5-2\beta}} =
c_\beta \int_0^t
L_r(B_r) \,dr.
\]
In order to prove~(\ref{eqL2-cvgce-bracket-Mh}), we now just have to
observe that
\[
\int_{0}^{t} L_r(B_r)
\,dr = \int_{0}^{t} \biggl(\int
_{0}^{r} \delta _{B_r}(B_u)
\,du \biggr) \,dr = \int_{0}^{t} \biggl(\int
_{0}^{r} \delta_{0}(B_r-B_u)
\,du \biggr) \,dr = \alpha_t.
\]
This completes our proof.\quad\qed
\end{longlist}\noqed
\end{pf}

\section{$L^2$ modulus of one-dimensional local time on chaoses}\label{sec1-dim}

In this section, we go back to the study of the $L^2$ modulus of the
Brownian local time, that is, to the study of the quantity
$H_{t}^{h}(B)$ defined by~(\ref{eqhamiltonian-mod-cty-loc-time}) with
the global aim of proving Theorem~\ref{thmcvgce-L2-modulus-chaos}.
Before we go on with the proof, let us introduce some additional notation.

%
\begin{notation}\label{notdef-simplex-Phi-h}
For any $t>0$ and $n\ge1$, we write $\mathcal S_t^{n}$ for the simplex of
order $n$ on $[0,t]$, that is, $\mathcal S_t^n=\{(t_1,\ldots,t_n)\in
[0,t]^n\dvtx
t_1<\cdots<t_n\}$. For every $n\geq2$ and every $h>0$, we also define
a function $\Phi_h(t_1,t_2)$ as
\[
\Phi_h(t_1,t_2)=\Phi_{h,n}(t_1,t_2):=
\int_0^h p_{t_2-t_1}^{(n-2)}(y)
(h-y) \,dy, \qquad0\leq t_1\leq t_2\leq t.
\]
\end{notation}

From the classical uniform estimate $\sup_{y\in\mathbb
R}|p_t^{(2m)}(y)| \leq
c_m t^{-m-(1/2)}$, we can already derive the following bounds on
$\Phi
_{h,2m}$, which will be used in the course of our reasoning.

%
\begin{lemma}
Fix $m\geq1$. Then there exists a constant $c_m$ such that for every
$h\in(0,1)$ and all $0\leq t_1 < s <t <t_2$, one has
%
\begin{equation}
\label{bound-Phi-1} \bigl|\Phi_{h,2m}(s,t)\bigr|\leq c_m h^2
\llvert t-s\rrvert ^{-m+(1/2)}
\end{equation}
and for any $\lambda\in(0,1)$,
%
\begin{eqnarray}
\label{bound-Phi-2} \bigl|\Phi_{h,2m}(t_1,t)-\Phi_{h,2m}(t_1,s)\bigr|
&\leq& c_m h^2 \llvert t-s \rrvert ^\lambda
\llvert s-t_1\rrvert ^{-m+(1/2)-\lambda},
\\
\label{bound-Phi-3} \bigl|\Phi_{h,2m}(t,t_2)-\Phi_{h,2m}(s,t_2)\bigr|
&\leq& c_m h^2 \llvert t-s \rrvert ^\lambda
\llvert t_2-t\rrvert ^{-m+(1/2)-\lambda}.
\end{eqnarray}
\end{lemma}

The proof of Theorem~\ref{thmcvgce-L2-modulus-chaos} is decomposed in
four main steps: after some preliminary material, we write an explicit
chaos decomposition for each $H_t^{h}(B)$. Then we study the asymptotic
behavior of the variance in each chaos, and the central limit theorem
for the finite-dimensional distributions of $J_n(H_.^{h}(B))$ is
obtained by analyzing the contractions of its sequence of kernels.
Finally, we study the tightness of the process $\{J_n(H_t^{h}(B));
t\in[0,1]\}$ properly normalized.

\subsection{Stochastic analysis preliminaries}\label{secprelim-stoch-analysis}
We will consider here the Brownian motion $B$ as an isonormal process
$B\equiv\{B(h); h\in\mathcal H\}$ defined on $(\Omega,\mathcal
F,\mathbf P)$, with $\mathcal H
=L^2([0,1])$. Recall that it means that $B$ is a centered Gaussian family
with covariance function $\mathbf E[B(h_1) B(h_2)]=\langle h_1,
h_2\rangle
_{\mathcal H}$. We also assume that $\mathcal F$ is generated by $B$.

At this point, we can introduce the Malliavin derivative operator on
the Wiener space $(\Omega,\mathcal H,\mathbf P)$. Namely, we first let
$\mathcal S$ be the
family of smooth functionals $F$ of the form
\[
F=f \bigl(B(h_1),\ldots,B(h_n) \bigr),
\]
where $h_1,\ldots,h_n\in\mathcal H$, $n\geq1$, and $f$ is a smooth function
having polynomial growth together with all its partial derivatives.
Then the Malliavin derivative of such a functional $F$ is the $\mathcal H
$-valued random variable defined by
\[
\mathcal DF= \sum_{i=1}^n
\frac{\partial f}{\partial x_i} \bigl(B(h_1),\ldots,B(h_n) \bigr)
h_i.
\]
For all $p>1$, it is known that the operator $\mathcal D$ is closable from
$L^p(\Omega)$ into $L^p(\Omega; \mathcal H)$.
We still denote by $\mathcal D$ the closure of this operator, whose
domain is
usually denoted by $\mathbb D^{1,p}$ and is defined
as the completion of $\mathcal S$ with respect to the norm
\[
\|F\|_{1,p}:= \bigl( \mathbf E \bigl[|F|^p \bigr]+
\mathbf E \bigl[\| \mathcal DF\|_\mathcal H^p \bigr]
\bigr)^{1/p}.
\]
We shall also denote by $\mathbb D^{\infty,p}$ the intersection $\bigcap_{k\ge
1}\mathbb D^{k,p}$.

Consider the $n$th Hermite polynomial $H_n$ defined
on $\mathbb R$, that is,
%
\begin{equation}
\label{eqdef-hermite-polynomial} H_n(x) = \frac{(-1)^{n} }{n!} e^{x^2/2}
\partial_{x}^{n} e^{-x^2/{2}}
\end{equation}
and let $\mathcal H_n$ be the closed linear subspace of $L^2(\Omega)$ generated
by the random variables $\{H_n(B(h)); h\in\mathcal H, \|h\|
_{\mathcal H}=1\}$.
Then $\mathcal H_n$ is called Wiener chaos of order~$n$, and
$L^2(\Omega)$ can
be decomposed into the orthogonal sum of the $\mathcal H_n$: we have
$L^2(\Omega,\mathcal F,\mathbf P) = \bigoplus_{n=0}^{\infty} \mathcal H_n$ (see
\cite{Nu-bk}, Theorem~1.1.1). In the sequel, we denote by $J_n(F)$ the projection of
a given random variable $F\in L^2(\Omega)$ onto $\mathcal H_n$ for
$n\ge0$,
with $J_0(F)=\mathbf E[F]$. In this context, Stroock's formula (see
\cite
{St}) states that, whenever $F\in\mathbb D^{\infty,2}$, one can compute
$J_n(F)$ explicitly as follows for $n\ge1$:
%
\begin{equation}
\label{eqstroock-formula} J_n(F) = I_n(f_n)\qquad
\mbox{with } f_n(t_{1},\ldots,t_{n}) =
\frac{\mathbf E [\mathcal
D_{t_1,\ldots,t_n} F  ]}{n!},
\end{equation}
where $I_n(f_n)$ stands for the multiple It\^o--Wiener integral of
$f_n$ with respect to $B$. We also label the value of $H_{2m}(0)$ here
for further use: for $m\ge1$, we have
%
\begin{equation}
\label{eqvalue-H-2k} H_{2m}(0) = \frac{(-1)^{m}}{2^{m} m!}.
\end{equation}

Let now $f_n$ be a symmetric function in $L^{2}([0,1]^{n})$. The
contraction of order $p$ of $f_n$ is the function defined on
$[0,1]^{2(n-p)}$ as follows:
\begin{eqnarray}
\label{eqdef-contraction}
&& [f_n\otimes_{p}
f_n ](t_{1},\ldots,t_{2(n-p)})\nonumber
\\
&&\qquad = \int_{[0,1]^{p}} f_{n}(u_1,
\ldots,u_p,t_1,\ldots,t_{n-p})
\\
&&\hspace*{59pt}{}\times f_{n}(u_1,\ldots,u_p,t_{n-p+1},
\ldots,t_{2(n-p)}) \,du_1\cdots \,du_p.\nonumber
\end{eqnarray}
With this definition in hand, let us state the following theorem
(borrowed from~\cite{NP}), which will be crucial in order to establish
the convergence of our renormalized local times.

%
\begin{proposition}\label{propfourth-moment}
Let $\{F_{h}= I_{n}(f_{n,h}); h>0\}$ be a family of random
variables belonging to a fixed Wiener chaos $\mathcal H_n$, for which we
assume that the kernels $f_{n,h}$ are symmetric. We also suppose that:
\begin{longlist}[(ii)]
\item[(i)] We have $\lim_{h\to0}\mathbf E[F_h^2]=\sigma^{2}>0$.

\item[(ii)] For all $p\in\{1,\ldots,n-1\}$, the relation $\lim_{h\to
0}\|f_{n,h}\otimes_{p} f_{n,h}\|_{\mathcal H^{\otimes2(n-p)}}=0$ holds true.

Then $F_h$ converges in law to a Gaussian random variable $\mathcal
N(0,\sigma
^2)$ as $h\to0$.
\end{longlist}
\end{proposition}

In order to obtain convergence in law for processes, we shall also
invoke a CLT for multidimensional vectors in a fixed chaos, originally
proved in \cite{peccati-tudor}:

%
\begin{proposition}\label{propfourth-moment-d-dim}
Consider a family of $d$-dimensional random variables $\{F_{h};
h>0\}$ with $F_{h}= (F_{h}^{1},\ldots,F_{h}^{d})$, such that
$F_{h}^{j}$ belongs to a fixed Wiener chaos $\mathcal H_n$ for each
$j\in\{
1,\ldots,d\}$ and $h>0$. Suppose furthermore that for a symmetric
matrix $\Gamma$ we have:
\begin{longlist}[(ii)]
\item[(i)] Each $F_{h}^{j}$ converges in law to a Gaussian random
variable $\mathcal N(0,\Gamma(i,i))$ as $h\to0$.

\item[(ii)] For each $(i,j)\in\{1,\ldots,d\}^{2}$, one has $\lim_{h\to
0}\mathbf E[F_{h}^{i} F_{h}^{j}]=\Gamma(i,j)$.

Then $F_h$ converges in law to a Gaussian random variable $\mathcal
N(0,\Gamma)$
as \mbox{$h\to0$}.
\end{longlist}
\end{proposition}

\subsection{Chaos decomposition of $H_t^{h}(B)$}
In order to compute the chaos decomposition of $H_t^{h}(B)$, we first
recall a relation taken from \cite{HN09}, whose proof is similar to our
identity~(\ref{eqH-gamma-f-beta}): we have
%
\begin{eqnarray}\label{eqrel-Ht-B-delta-0}
H_t^{h}(B)
&=& \int_{[0,t]^{2}}
\bigl[\delta_0(B_v-B_u+h)
\nonumber\\[-8pt]\\[-8pt]
&&\hspace*{26pt}{} +
\delta_0(B_v-B_u-h)-2
\delta_0(B_v-B_u) \bigr] \,du \,dv,\nonumber
\end{eqnarray}
where $\delta_0(B_v-B_u+h)$ has to be understood as a distribution on the
Wiener space in the sense of Watanabe (see \cite{Wa}). One can also
show that the right-hand side of~(\ref{eqrel-Ht-B-delta-0}) is the
$L^2(\Omega)$-limit of the sequence obtained by replacing $\delta_0$ with
the Gaussian approximating kernel $p_\varepsilon$ (see \cite{HN09}, Section~2,
for further details).

Let us also give an elementary yet useful lemma.

%
\begin{lemma}\label{lemexpected-pn-Gaussian}
Let $p_t$ be the Gaussian kernel defined by (\ref
{eqdef-gauss-kernel}), and $N$ be a real valued random variable such
that $N\sim\mathcal N(h,\sigma^2)$ with $h\in\mathbb R$ and $\sigma
^2>0$. Then for all
$n\ge0$, we have
%
\begin{equation}
\label{eqexpect-ptn-N} \mathbf E \bigl[p_t^{(n)}(N) \bigr]=
p_{t+\sigma^2}^{(n)}(h).
\end{equation}
\end{lemma}

\begin{pf}
Taking into account the analytic form of expected values with respect
to $N$, we have $\mathbf E[ p_t^{(n)}(N) ]=[p_t^{(n)}*p_{\sigma^2}](h)$.
Furthermore, elementary relations for convolutions and the semigroup
property for $p$ yield:
\[
p_t^{(n)}*p_{\sigma^2} = [p_t*p_{\sigma^2}
]^{(n)} = p_{t+\sigma^2}^{(n)}
\]
from which relation~(\ref{eqexpect-ptn-N}) is easily deduced.
\end{pf}

Recall now that the projection $J_n(F)$ of a $L^2$ random variable $F$
onto a fixed chaos $\mathcal H_n$ has been defined at Section~\ref
{secprelim-stoch-analysis}. For our Hamiltonian $H_t^{h}(B)$, we get
the following.

%
\begin{proposition}\label{propproj-1-d}
For every $n\geq1$ and every $h>0$, recall that we have set
$X^{n,h}_{t}= J_{n}(H^{h}_{t}(B))$ for the projection of $H_t^{h}(B)$
onto the $n$th Wiener chaos. Then we have
%
\begin{eqnarray}\label{eqexp-Jn-HtB}
X^{n,h}_{t}&=&0\qquad\mbox{if $n$ is odd},
\nonumber\\[-8pt]\\[-8pt]
X^{n,h}_{t}&=&\frac{16}{n!} I_n
\bigl((f_h+g_{h,t})\cdot\mathbf {1}_{[0,t]^n} \bigr)
\qquad\mbox{if $n$ is even},\nonumber
\end{eqnarray}
where $f_h\in L^2(\mathbb R_+^n)$, $g_{h,t}\in L^2([0,t]^n)$ are the symmetric
functions defined by
%
\begin{eqnarray}
\label{definf-h} f_h(t_1,\ldots,t_n)&:=&
\Phi_h \bigl(\min(t_1,\ldots,t_n),
\max(t_1,\ldots,t_n) \bigr),
\\
g_{h,t}(t_1,\ldots,t_n)&:=&-
\Phi_h \bigl(\min(t_1,\ldots,t_n),t \bigr)+
\Phi _h(0,t)
\nonumber\\[-8pt]\label{defing-h} \\[-8pt]
&&{} -\Phi_h \bigl(0,\max(t_1,
\ldots,t_n) \bigr)\nonumber
\end{eqnarray}
and where we recall that the function $\Phi_h$ has been defined at
Notation~\ref{notdef-simplex-Phi-h}.
\end{proposition}

\begin{pf}
We divide this proof in two steps:
\begin{longlist}
\item[\textit{Step} 1.]
\textit{Computation of the projection.}
Let us first compute the chaos decomposition of $\delta_0(B_v-B_u+h)$.
To this aim, recall that, as a distribution on the Wiener space (see
\cite{Wa}), we have $\delta_0(B_v-B_u+h)=\lim_{\varepsilon\to0}
p_{\varepsilon
}(B_v-B_u+h)$ for the Gaussian kernel $p_{\varepsilon}$ defined at
(\ref{eqdef-gauss-kernel}). Furthermore, according to Stroock's formula~(\ref{eqstroock-formula}), we have $J_{n}(p_{\varepsilon
}(B_v-B_u+h))=I_n(\varphi
_n^{\varepsilon})$ with
\begin{eqnarray*}
\varphi_n^{\varepsilon}(t_1,\ldots,t_n) &=&
\frac{1}{n!} \mathbf E \bigl[\mathcal D_{t_1,\ldots,t_n} p_{\varepsilon}(B_v-B_u+h)
\bigr]
\\
& =& \frac{1}{n!} \mathbf E \bigl[p_{\varepsilon}^{(n)}(B_v-B_u+h)
\bigr]\prod_{i=1}^{n} \mathbf{1}
_{[u,v]}(t_i). 
\end{eqnarray*}
We now compute $\mathbf E[ p_{\varepsilon}^{(n)}(B_v-B_u+h) ]$ by
means of formula
(\ref{eqexpect-ptn-N}), which yields
\[
\varphi_n^{\varepsilon}(t_1,\ldots,t_n) =
\frac{p_{v-u+\varepsilon}^{(n)}(h)}{n!} \prod_{i=1}^{n} \mathbf
{1}_{[u,v]}(t_i).
\]
Taking limits as $\varepsilon\to0$, we end up with $J_{n}(\delta
_0(B_v-B_u+h))=I_n(\varphi_n)$, where
\[
\varphi_n(t_1,\ldots,t_n) =
\frac{p_{v-u}^{(n)}(h)}{n!} \prod_{i=1}^{n}
\mathbf{1}_{[u,v]}(t_i).
\]
The same kind of computations is valid for $\delta_0(B_v-B_u-h)$ and
$\delta
_0(B_v-B_u)$, and thus going back to~(\ref{eqrel-Ht-B-delta-0}), we
have obtained
\begin{eqnarray*}
X^{n,h}_{t} &=& J_n \bigl(H_t^{h}(B)
\bigr)
\\
&=& \frac{2}{n!} I_n \Biggl( \int_{\mathcal S_{t}^{2}}
\prod_{i=1}^{n} \mathbf{1}_{[u,v]}(t_i)
\bigl[p_{v-u}^{(n)}(h) + p_{v-u}^{(n)}(-h)
- 2 p_{v-u}^{(n)}(0) \bigr] \,du \,dv \Biggr),
\end{eqnarray*}
where we recall that $\mathcal S_t^2$ stands for the simplex of order
$2$ on
$[0,t]$ (see Notation~\ref{notdef-simplex-Phi-h}). Moreover, observe
that $p_{v-u}^{(n)}(h) + p_{v-u}^{(n)}(-h) - 2 p_{v-u}^{(n)}(0)\equiv
0$ when $n$ is odd, which yields the first claim in (\ref
{eqexp-Jn-HtB}). Therefore, only even $n$s are considered from now on.
\end{longlist}
\begin{longlist}
\item[\textit{Step} 2.]
\textit{Simplification of the expression for the projection.}
Notice first that, since we are dealing with a linear Brownian motion
$B$, one can write $X^{n,h}_{t}$ as
%
\begin{eqnarray}
\label{eqJn-HtB-2} X^{n,h}_{t} &=& 2 \int_{\mathcal S_{t}^{n}}
\Biggl(\int_{\mathcal S_{t}^{2}} \prod_{i=1}^{n}
\mathbf{1}_{[u,v]}(t_i) \bigl[p_{v-u}^{(n)}(h)+ p_{v-u}^{(n)}(-h)\nonumber
\\
&&\hspace*{144pt}{}  - 2 p_{v-u}^{(n)}(0)
\bigr] \,du \,dv \Biggr)\,dB_{t_1} \cdots \,dB_{t_n}
\\
&=& 2 \int_{\mathcal S_{t}^{n}} \biggl(\int_{0}^{t_1}\!\!
\int_{t_n}^{t} \bigl[p_{v-u}^{(n)}(h)
+ p_{v-u}^{(n)}(-h) - 2 p_{v-u}^{(n)}(0)
\bigr] \,dv \,du \biggr)\,dB_{t_1} \cdots \,dB_{t_n}.\nonumber\hspace*{-10pt}
\end{eqnarray}
Let us transform now the expression $p_{v-u}^{(n)}(h) +
p_{v-u}^{(n)}(-h) - 2 p_{v-u}^{(n)}(0)$. First, since $n$ is an even
number and $p$ is symmetric, we have
\[
p_{v-u}^{(n)}(h) + p_{v-u}^{(n)}(-h) - 2
p_{v-u}^{(n)}(0) = 2 \bigl[p_{v-u}^{(n)}(h)
- p_{v-u}^{(n)}(0) \bigr].
\]
Then write
\begin{eqnarray*}
p_{v-u}^{(n)}(h) - p_{v-u}^{(n)}(0) &=&
\int_{0}^{h} p_{v-u}^{(n+1)}(x)
\,dx = \int_{0}^{h}\!\! \int_{0}^{x}
p_{v-u}^{(n+2)}(y) \,dy \,dx
\\
&=& 2 \int_{0}^{h}\!\!
\int_{0}^{x} \partial_{v}
p_{v-u}^{(n)}(y) \,dy \,dx,
\end{eqnarray*}
which yields
\begin{eqnarray*}
&& \int_{0}^{t_1} du \int_{t_n}^{t}
dv \bigl[p_{v-u}^{(n)}(h) - p_{v-u}^{(n)}(0)
\bigr]
\\
&&\qquad = 2 \int_{0}^{t_1} du \int
_{t_n}^{t} dv \int_{0}^{h}
dx \int_{0}^{x} dy \, \partial_{v}
p_{v-u}^{(n)}(y)
\\
&&\qquad = 2 \int_{0}^{t_1} du \int_{0}^{h}
dx \int_{0}^{x} dy \bigl[ p_{t-u}^{(n)}(y)
- p_{t_n-u}^{(n)}(y) \bigr]
\\
&&\qquad= -4 \int_{0}^{t_1} du \int
_{0}^{h} dx \int_{0}^{x}
dy \bigl[\partial _{u} p_{t-u}^{(n-2)}(y) -
\partial_{u} p_{t_n-u}^{(n-2)}(y) \bigr]
\\
&&\qquad= -4 \int_{0}^{h} \bigl[ p_{t-t_1}^{(n-2)}(y)
- p_{t_n-t_1}^{(n-2)}(y) - p_{t}^{(n-2)}(y) +
p_{t_n}^{(n-2)}(y) \bigr](h-y) \,dy.
\end{eqnarray*}
Plugging this expression into~(\ref{eqJn-HtB-2}) and symmetrizing
again, relation~(\ref{eqexp-Jn-HtB}) easily follows.\quad\qed
\end{longlist}\noqed
\end{pf}

\subsection{Asymptotic behavior of the variance}\label{secrenorm}
In this section, we compute the correct amount of normalization needed
for the convergence of each $X^{2m,h}_{t}=J_{2m}(H_{t}(B))$ for $m\ge
1$. This will be obtained thanks to an asymptotic analysis of the
variance of those random variables and recall that we have shown that
$X^{2m,h}_{t}=\frac{16}{(2m)!} I_{2m} ((f_h+g_{h,t})\cdot\mathbf{1}
_{[0,t]^{2m}} )$, which means in particular that
%
\begin{eqnarray}\label{isom-int}
&& \mathbf E \bigl[ X^{2m,h}_t
X^{2m,h}_s \bigr]
\nonumber\\[-8pt]\\[-8pt]
&&\qquad =\frac{16^2}{(2m)!} \bigl\langle
(f_h+g_{h,t})\cdot\mathbf{1}_{[0,t]^{2m}},(f_h+g_{h,s})
\cdot\mathbf {1}_{[0,s]^{2m}} \bigr\rangle_{L^2(\mathbb R_+^{2m})}.\nonumber
\end{eqnarray}
Our aim is to prove the following.

%
\begin{proposition}\label{proprenorm}
Fix $m \geq1$. Then for all $0\leq s\leq t\leq1$, it holds that
%
\begin{equation}
\label{eqrenorm-var-Jn-HtB} \lim_{h\to0}\frac{\mathbf E [ X^{2m,h}_t X^{2m,h}_s  ]}{h^4
\ln(1/ h)} =
\sigma_m^2 s \qquad\mbox{with } \sigma_m^2=
\frac{c
(2m-2)!}{2^{2m} [(m-1)!]^2}
\end{equation}
for some strictly positive universal constant $c$.
\end{proposition}

The strategy for the proof of Proposition~\ref{proprenorm} is rather
simple. Namely, with the expression~(\ref{isom-int}) in mind, our
calculations will be decomposed into the following facts:
\begin{itemize}
\item
The norm $\|g_{h,t}\|_{L^2([0,t]^{2m})}^2$ is of order at most $h^4$ as
$h$ tends to $0$, and thus is negligible with respect to $h^4 \ln(1/h)$.
\item
The quantity $\langle f_h \cdot\mathbf{1}_{[0,t]^{2m}},f_h\cdot
\mathbf{1}
_{[0,s]^{2m}} \rangle_{L^2(\mathbb R_+^{2m})}$ scales as in relation
(\ref{eqrenorm-var-Jn-HtB}).
\end{itemize}
Let us thus start by identifying the negligible terms.

%
\begin{lemma}\label{lemmag-h}
Fix $m\geq1$, and recall that for every $t> 0$, $g_{h,t}=g_{h,t,2m}$
is defined by~(\ref{defing-h}). Then there exists a constant $c_m$
such that for every $h>0$,
\[
\sup_{t\in[0,1]}\|g_{h,t}\|_{L^2([0,t]^{2m})}^2
\leq c_m h^4.
\]
\end{lemma}

\begin{pf}
Write
\begin{eqnarray*}
&& \|g_{h,t}\|_{L^2([0,t]^{2m})}^2
\\
&&\qquad =(2m)! \int_{\mathcal S_t^{2m}} \bigl[ \Phi_h(t_1,t)-
\Phi _h(0,t)+\Phi _h(0,t_{2m})
\bigr]^2 \,dt_1 \cdots \,dt_{2m}
\\
&&\qquad \leq c_m \biggl\{\int_{[0,t]}
(t-t_1)^{2m-1} \Phi_h(t_1,t)^2
\,dt_1
\\
&&\hspace*{49pt}{} +t^{2m} \Phi_h(0,t)^2+\int
_{[0,t]} t_{2m}^{2m-1} \Phi_h(0,t_{2m})^2
\,dt_{2m} \biggr\}
\end{eqnarray*}
and the bound is then easily derived from~(\ref{bound-Phi-1}).
\end{pf}

We can now turn to the proof of the main proposition of this section.

\begin{pf*}{Proof of Proposition~\ref{proprenorm}}
Thanks to Lemma~\ref{lemmag-h}, we only have to focus on
\[
A_h(s,t)\equiv\langle f_h\cdot\mathbf{1}_{[0,t]^{2m}},f_h
\cdot \mathbf{1} _{[0,s]^{2m}} \rangle_{L^2(\mathbb R_+^{2m})}.
\]
An easy integration over the simplex gives
\begin{eqnarray*}
A_h(s,t)&=& (2m)! \int_{\mathcal S_{s}^{2m}} \bigl[
\Phi_{h,2m}(t_1,t_{2m}) \bigr]^2
\,dt_1 \cdots \,dt_{2m}
\\
&=& (2m)! \int_{\mathcal S_{s}^{2}} \frac{(t_{2m}-t_1)^{{2m}-2}}{({2m}-2)!} \bigl[
\Phi_{h,2m}(t_1,t_{2m}) \bigr]^2
\,dt_1 \,dt_{2m}.
\end{eqnarray*}
Then, using the classical formula for the $2m$th derivative of $p_t$,
that is,
%
\begin{equation}
\label{formulderiv} p_t^{(2m)}(y) = (2m)! t^{-m}
p_t(y) H_{2m} \biggl(\frac{y}{t^{1/2}} \biggr),
\end{equation}
where $H_{2m}$ is defined by~(\ref{eqdef-hermite-polynomial}), we
deduce that
\begin{eqnarray*}
A_h(s,t) &=& (2m)! ({2m}-2)!
\\
&&{}\times \int_{\mathcal S_{s}^{2}}
(t_{2}-t_1)^{{2m}-2}
\biggl[\int_{0}^{h}
(t_{2}-t_1)^{-(m-1)} \frac
{e^{-{y^2}/({2(t_{2}-t_1)})}}{(2\pi(t_{2}-t_1))^{1/2}}
\\
&&\hspace*{105pt}{}\times
H_{{2m}-2} \biggl( \frac{y}{ (t_{2}-t_1)^{1/2}} \biggr) (h-y) \,dy
\biggr]^2 \,dt_1 \,dt_{2}
\\
& =& (2m)! (2m-2)!
\\
&&{}\times  \int_{\mathcal S_{s}^{2}}\biggl[\int
_{0}^{h} \frac{e^{-y^2/({2(t_{2}-t_1)})}}{(2\pi
(t_{2}-t_1))^{1/2}} H_{{2m}-2}
\biggl(\frac{y}{
(t_{2}-t_1)^{1/2}} \biggr) (h-y) \,dy \biggr]^2 \,dt_1
\,dt_{2}.
\end{eqnarray*}
Perform the change of variable $t_{2}-t_1=\tau$ and $t_1=\sigma$,
which yields
\begin{eqnarray*}
A_h(s,t) &=& (2m)! (2m-2)!
\\
&&{}\times  \int_{0}^{s}
(s-\tau) \biggl[\int_{0}^{h} \frac{e^{-y^2/(2\tau)}}{(2\pi\tau)^{1/2}}
H_{{2m}-2} \biggl(\frac{y}{ \tau^{1/2}} \biggr) (h-y) \,dy
\biggr]^2 \,d\tau.
\end{eqnarray*}
Now set $y/\tau^{1/2}=z$ in order to get
\begin{eqnarray*}
A_h(s,t)&=& \frac{(2m)! (2m-2)!} {2\pi} h^2
\\
&&{}\times  \int
_{0}^{s} (s-\tau) \biggl[\int
_{0}^{h/\tau^{1/2}} e^{-z^2/2} H_{{2m}-2} (z
) \biggl(1-\frac{\tau^{1/2} z}{h} \biggr) \,dz \biggr]^2 \,d\tau.
\end{eqnarray*}
Finally, let $u=h/\tau^{1/2}$, so that we end up with $A_h(s,t)= \frac{1} {\pi}(2m)! (2m-2)! h^4 a(h)$, where
\[
a(h)\equiv \int_{h/s^{1/2}}^{\infty} u^{-3}
\biggl(s-\frac{h^2}{u^2} \biggr) \biggl[\int_{0}^{u}
e^{-z^2/2} H_{{2m}-2} (z ) \biggl(1-\frac
{z}{u} \biggr) \,dz
\biggr]^2 \,du.
\]
It is now readily checked that the main singularity in the integral
defining $a(h)$ is due to a term $u^{-3}u^{2}=u^{-1}$ integrated close
to 0, so that for small $h$, $a(h)$ is of order $\ln(1/h)$.

In order to quantify this fact, let us apply l'Hopital's rule to
$a(h)/\ln(1/h)$. We get
\[
\lim_{h\to0} \frac{a(h)}{\ln(1/h)} = \lim_{h\to0}
\frac
{b(h)}{h^{-2}}
\]
with
\[
b(h)= 2 \int_{h/s^{1/2}}^{\infty}
u^{-5} \biggl[\int_{0}^{u}
e^{-z^2/2} H_{{2m}-2} (z ) \biggl(1-\frac
{z}{u} \biggr) \,dz
\biggr]^2 \,du.
\]
It is now easily seen that $b'(h)$ is equivalent to $-\frac{s}{2}
h^{-3} [H_{2m-2}(0)]^{2}$ in a neighborhood of the origin, so that a
second application of l'Hopital's rule to $b(h)/h^{-2}$ yields
\[
\lim_{h\to0} \frac{a(h)}{\ln(1/h)} = \frac{s}{4}
\bigl[H_{2m-2}(0) \bigr]^{2}.
\]
In order to conclude recall that $A_h(s,t)= \frac{1}{\pi} (2m)! (2m-2)!
h^4 a(h)$, and thus with the value of $H_{2m-2}(0)$ in mind [see
(\ref{eqvalue-H-2k})] we end up with
\[
\lim_{h\to0} \frac{A_h(s,t)}{h^4\ln(1/h)} = \frac{s}{4} \bigl[
H_{2m-2}(0) \bigr]^2 (2m)! (2m-2)! = \frac{(2m)!(2m-2)!}{2^{2m-1}[(m-1)!]^2} s,
\]
which completes the proof of relation~(\ref{eqrenorm-var-Jn-HtB}) since
\[
\lim_{h\to0}\frac{\mathbf E [ X^{2m,h}_t X^{2m,h}_s  ]}{h^4
\ln(1/
h)}=\frac{128}{\pi(2m)!}\lim
_{h\to0} \frac{A_h(s,t)}{h^4\ln(1/h)}.
\]\upqed
\end{pf*}

%
\begin{remark}
The fact that $\sum\sigma_m^2=\infty$, mentioned at Theorem~\ref{thmcvgce-L2-modulus-chaos}(iii), follows at once from relation
(\ref{eqrenorm-var-Jn-HtB}). Indeed, using Stirling's formula, we can
easily conclude that $\sigma_m^2$ is asymptotically equivalent to
$\frac
{c}{\sqrt{m}}$ for some constant $c>0$.
\end{remark}

\subsection{Contractions}\label{seccontractions-1-d}
In this section, we shall prove that for a fixed $t\in[0,1]$ the random
variable $H_{t}(B)/[h^2\ln(1/h)^{1/2}]$ converges in law to a Gaussian
random variable as $h$ goes to 0. Owing to Proposition~\ref{propfourth-moment} and with Proposition~\ref{proprenorm} in hand,
this boils down to the study of contractions for the functions
$f_h,g_h$ involved in the definition of $J_n(H_{t}(B))$ given at (\ref
{eqexp-Jn-HtB}). Those contractions are evaluated in the following proposition.

%
\begin{proposition}\label{propcontraction-f-h-plus-g-h}
Fix $n=2m \geq2$, and recall that $f_h,g_{h,t}$ also depend on $n$ as
highlighted in~(\ref{definf-h})--(\ref{defing-h}). Then for every
$r\in\{1,\ldots,n-1\}$, one has
%
\begin{equation}
\label{eqcontraction-f-h-plus-g-h} \frac{1}{h^8 \ln^2(1/h)} \bigl\|(f_h+g_{h,t})
\otimes_r (f_h+g_{h,t})\bigr\|
_{L^2([0,t]^{2n-2r})}^2 \to0
\end{equation}
as $h$ tends to $0$.
\end{proposition}

\begin{pf}
Due to Lemma~\ref{lemmag-h}, the proof of Proposition~\ref{proprenorm} and thanks to the fact that
\[
\|f_h \otimes_r g_{h,t}
\|_{L^2([0,t]^{2n-2r})} \leq\|f_h\| _{L^2([0,t]^{2m})} \|g_{h,t}
\|_{L^2([0,t]^{2m})},
\]
it is readily checked that as $h$ tends to $0$,
\begin{eqnarray*}
&& \frac{1}{h^8 \ln^2(1/h)} \bigl\|(f_h+g_{h,t}) \otimes
_r (f_h+g_{h,t})\bigr\|
_{L^2([0,t]^{2n-2r})}^2
\\
&&\qquad =\frac{1}{h^8 \ln^2(1/h)} \|f_h \otimes
_r f_h\| _{L^2([0,t]^{2n-2r})}^2+o(1).
\end{eqnarray*}
We are thus reduced to prove that
%
\begin{equation}
\label{eqcontraction-f-h} \lim_{h\to0}\frac{ \|f_h \otimes_r f_h\|_{L^2([0,t]^{2n-2r})}^2}{h^8
\ln^2(1/h)} = 0.
\end{equation}
In order to compute $\|f_h \otimes_r f_h\|_{L^2([0,t]^{2n-2r})}^2$,
let us consider the following general problem: fix an integrable
function $\varphi$ defined on $\mathcal S_{t}^2$
and compute the contraction norm:
\begin{eqnarray*}
R_{n,r}(\varphi) &=& \int_{[0,t]^{2(n-r)}} \biggl(\int
_{[0,t]^{r}} \varphi \bigl(\max \bigl(\mathbf{s},
\mathbf{t}^{1} \bigr),\min \bigl(\mathbf{s},\mathbf
{t}^{1} \bigr) \bigr)
\\
&&\hspace*{72pt} {}\times \varphi \bigl( \max \bigl(\mathbf{s},
\mathbf{t}^{2} \bigr),\min \bigl(\mathbf{s}, \mathbf{t}^{2}
\bigr) \bigr) \,d\mathbf{s} \biggr)^{2} \,d\mathbf {t}^{1} \,d
\mathbf{t}^{2},
\end{eqnarray*}
where we have set
\begin{eqnarray*}
\max(\mathbf{s},\mathbf{t})&=&\max(s_1,\ldots,s_r,t_1,
\ldots,t_{n-r}),
\\
\min(\mathbf{s},\mathbf{t})&=&
\min(s_1,\ldots,s_r,t_1,\ldots,t_{n-r}).
\end{eqnarray*}
Note that $R_{n,r}(\varphi)$ can also be written as
\[
R_{n,r}(\varphi) = \int_{[0,t]^{2(n-r)}} \int
_{[0,t]^{2r}} \prod_{i,j=1}^{2}
\varphi \bigl(\max \bigl(\mathbf{s}^{i},\mathbf {t}^{j}
\bigr),\min \bigl(\mathbf{s}^{i},\mathbf{t}^{j} \bigr)
\bigr) \,d\mathbf{s}^{1} \,d\mathbf{s}^{2} \,d
\mathbf{t}^{1} \,d\mathbf{t}^{2}.
\]
In order to evaluate this integral, the following simple
transformations can be performed: \textup{(i)} Replace $\max(\mathbf
{s}^{i},\mathbf{t}
^{j})$ by $\max(\mathbf{s}^{i})\vee\max(\mathbf{t}^{j})$. \textup{(ii)} Integrate on
simplexes such as $0<s_1<\cdots<s_r<t$. For $2\leq r\leq n-2$; this
simplifies the above expression into
\begin{eqnarray*}
R_{n,r}(\varphi)&=& \bigl[(n-r)! r! \bigr]^{2}
\\
&&\!{}\times \int
_{(\mathcal S_{t}^2)^{2}} \int_{(\mathcal S_{t}^{2})^{2}} \prod
_{i,j=1}^{2} \varphi \bigl(\max \bigl(
\sigma_{2}^{i},\tau _{2}^{j} \bigr),
\min \bigl(\sigma _{1}^{i},\tau_{1}^{j}
\bigr) \bigr)
\\
&&\!\hspace*{63pt}{}\times\prod_{k=1}^{2} \biggl[
\frac{(\sigma_2^{k}-\sigma
_1^{k})^{r-2} (\tau
_2^{k}-\tau_1^{k})^{n-r-2}}{(r-2)! (n-r-2)!} \biggr] \,d\sigma_1^{k} \,d
\sigma_2^{k} \,d\tau_1^{k} \,d
\tau_2^{k},
\end{eqnarray*}
that is,
\begin{eqnarray}
\label{eqdef-R-n-r}
R_{n,r}(\varphi)&=& P_{r}(n)\nonumber\hspace*{-5pt}
\\
&&{}\times  \int
_{(\mathcal S_{t}^2)^{2}} \int_{(\mathcal S_{t}^{2})^{2}} \prod
_{i,j=1}^{2} \varphi \bigl(\max \bigl(
\sigma_{2}^{i},\tau _{2}^{j} \bigr),
\min \bigl(\sigma _{1}^{i},\tau_{1}^{j}
\bigr) \bigr)\hspace*{-5pt}
\\
&&\hspace*{63pt}{}\times\prod_{k=1}^{2} \bigl(
\sigma_2^{k}-\sigma_1^{k}
\bigr)^{r-2} \bigl(\tau _2^{k}-\tau
_1^{k} \bigr)^{n-r-2} \,d\sigma_1^{k}
\,d\sigma_2^{k} \,d\tau_1^{k} \,d
\tau_2^{k},\nonumber\hspace*{-5pt}
\end{eqnarray}
where we have set $P_{r}(n)=[ (n-r) (n-r-1) r (r-1) ]^{2}$.

We now recall that $f_h$ is defined by~(\ref{definf-h}), which means
that we shall apply identity~(\ref{eqdef-R-n-r}) to the function
$\varphi
=\Phi_h$ where $\Phi_h$ is introduced at Notation~\ref{notdef-simplex-Phi-h}. Toward this aim, observe that one can write
$\Phi_h(u,v)=\ell_{n,h}(v-u)$ with $\ell_{n,h}\dvtx \mathbb R_+\to
\mathbb R_+$ given by
$\ell_{n,h}(w):= \int_{0}^{h} p_{w}^{(n-2)}(y) (h-y) \,dy$. Thanks
to the expression~(\ref{formulderiv}) we have already recalled for $
p_{w}^{(n-2)}$ we thus get
\[
\ell_{n,h}(w)=\frac{(-1)^n(n-2)! } {\sqrt{2\pi} } \int_{0}^{h}
\frac{e^{-y^2/(2w)}}{w^{(n-1)/2}} H_{n-2} \biggl(\frac{y}{
w^{1/2}} \biggr) (h-y) \,dy
\leq c_n \frac{h^2}{w^{(n-1)/2}}.
\]
Plugging this relation into~(\ref{eqdef-R-n-r}), we obtain that for
$2\leq r\leq n-2$,
%
\begin{eqnarray}
\label{eqbnd-R-n-r} && \| f_h \otimes_r
f_h\|_{L^2([0,t]^{2n-2r})}^2\nonumber\hspace*{-10pt}
\\
&&\qquad \le c_{n,r} h^8 \int_{(\mathcal S_{t}^2)^{2}} \int
_{(\mathcal S_{t}^{2})^{2}} \prod_{i,j=1}^{2}
\bigl(\max \bigl(\sigma_{2}^{i},\tau_{2}^{j}
\bigr)-\min \bigl(\sigma _{1}^{i},\tau_{1}^{j}
\bigr) \bigr)^{-(n-2)/2}
\nonumber\hspace*{-10pt}
\\
&&\hspace*{79pt}\quad\qquad{}\times\prod_{k=1}^{2} \bigl(
\sigma_2^{k}-\sigma_1^{k}
\bigr)^{r-2} \bigl(\tau _2^{k}-\tau
_1^{k} \bigr)^{n-r-2} \,d\sigma_1^{k}
\,d\sigma_2^{k} \,d\tau_1^{k} \,d
\tau_2^{k}\hspace*{-10pt}
\\
&&\qquad \le c_{n,r} h^8 \int_{(\mathcal S_{t}^2)^{2}} \int
_{(\mathcal S_{t}^{2})^{2}} \prod_{i,j=1}^{2}
\bigl(\max \bigl(\sigma_{2}^{i},\tau_{2}^{j}
\bigr)-\min \bigl(\sigma _{1}^{i},\tau_{1}^{j}
\bigr) \bigr)^{-3/2}\nonumber\hspace*{-10pt}
\\
&&\hspace*{112pt}{}\times  \prod_{k=1}^{2}
\,d\sigma_1^{k} \,d\sigma_2^{k} \,d
\tau_1^{k} \,d\tau_2^{k}.\nonumber\hspace*{-10pt}
\end{eqnarray}
This kind of integral will be handled in Lemma~\ref{lemtechni}, which
allows to conclude that $\| f_h \otimes_r f_h\|_{L^2([0,t]^{2n-2r})}^2
\le c_{n,r} h^8$. Hence, relation~(\ref{eqcontraction-f-h})
obviously holds true, which in turn implies~(\ref{eqcontraction-f-h-plus-g-h}).

We have thus proved relation~(\ref{eqcontraction-f-h-plus-g-h}) for
$n\ge4$ and $2\le r\le n-2$. The remaining possibilities can be
treated applying the same reasoning: in the case $(n\geq4, r\in\{
1,n-1\})$, we have
\begin{eqnarray*}
&& \|f_h \otimes_r f_h \|_{L^2([0,t]^{2n-2r})}^2
\\
&&\qquad \leq c_{r,n} h^8
\\
&&\quad\qquad{}\times  \int_{(\mathcal S_{t}^2)^2}\int
_{[0,t]^2} \prod_{i,j=1}^2
\bigl(\max \bigl(\sigma^i,\tau_2^j \bigr)-
\min \bigl(\sigma^i,\tau_1^j \bigr)
\bigr)^{-1} \,d\sigma^1 \,d\sigma^2\prod
_{k=1}^2 \,d\tau_1^k\,d\tau_2^k
\end{eqnarray*}
and we recognize here the second (finite) integral involved in Lemma
\ref{lemtechni}. Finally, the case $(n=2, r=1)$ reduces to
\begin{eqnarray*}
&& \|f_h \otimes_r f_h
\|_{L^2([0,t]^{2n-2r})}^2
\\
&&\qquad \leq c_{r,n} h^8
\\
&&\quad\qquad {}\times \int_{[0,t]^2} \int
_{[0,t]^2} \prod_{i,j=1}^2
\bigl(\max \bigl(\sigma^i,\tau^j \bigr)-\min \bigl(
\sigma^i,\tau^j \bigr) \bigr)^{-1/2} \,d
\sigma^1 \,d\sigma^2 \,d\tau^1 \,d
\tau^2,
\end{eqnarray*}
so that we can conclude with Lemma~\ref{lemtechni} as well.
\end{pf}

Summarizing our considerations up to now, we have obtained the
following convergence in law for the finite-dimensional distributions
of $X^{2m,h}$:

%
\begin{proposition}\label{proppointwise-cvgce-X-2m}
Taking up the notation of Theorem~\ref{thmcvgce-L2-modulus-chaos},
consider $t_1,\ldots,t_d\in[0,1]$ and $m\ge1$. Then as $h\to0$ we have
\begin{eqnarray}
\frac{1}{h^{2} [\ln(1/h)]^{1/2}} \bigl(X^{2m,h}_{t_1},\ldots,X^{2m,h}_{t_d}
\bigr) \xrightarrow{(d)} \sigma_{m} \mathcal N(0,\Gamma)\nonumber
\\
\eqntext{\displaystyle\mbox{where } \sigma_{m}^{2}= \frac{c (2m-2)!}{2^{2m} [(m-1)!]^2}}
\end{eqnarray}
and $\mathcal N(0,\Gamma)$ is the centered Gaussian law in $\mathbb
R^d$ with
covariance matrix $\Gamma(i,j)=\min(t_i,t_j)$.
\end{proposition}

\begin{pf}
We shall simultaneously apply Propositions~\ref{propfourth-moment} and
\ref{propfourth-moment-d-dim} to the random vector $h^{-2} [\ln
(1/h)]^{-1/2} (X^{2m,h}_{t_1},\ldots,X^{2m,h}_{t_d})$, which is of
course a sequence of random vectors in $(\mathcal H_{2m})^d$. Moreover:

(i) According to Proposition~\ref{proprenorm}, we have for
every $t_i,t_j$,
\[
\lim_{h\to0} \frac{\mathbf E[X^{2m,h}_{t_i}X^{2m,h}_{t_j}]}{h^{4}
\ln(1/h)} = \sigma_{m}^{2}
\Gamma(i,j).
\]

(ii) For each fixed $t_i$, one can write $X^{2m,h}_{t_i} =
I_{2m}(k_{2m,h})$ with $k_{2m,h}=\frac{16}{(2m)!} (f_h+g_{h,t_i})\cdot
\mathbf{1}_{[0,t_i]^{2m}}$. Then Proposition~\ref{propcontraction-f-h-plus-g-h} asserts that, for all $r\in\{1,\ldots,2m-1\}$,
\[
\lim_{h\to0} \frac{1}{h^4 \ln(1/h)} \|k_{2m,h} \otimes
_r k_{2m,h}\| _{\mathcal H^{\otimes2(n-r)}} =0.
\]
We can thus combine Propositions~\ref{propfourth-moment} and~\ref{propfourth-moment-d-dim} so as to conclude.
\end{pf}

\subsection{Tightness}

Now endowed with Proposition~\ref{proppointwise-cvgce-X-2m}, the proof
of Theorem~\ref{thmcvgce-L2-modulus-chaos} reduces to showing that the
sequence of processes $\{h^{-2} \ln(1/h)^{-1/2} \*X^{2m,h}_t, t\in
[0,1]\}$ is tight. These are contents of the following proposition.

%
\begin{proposition}\label{propotight}
Fix $m\geq1$. Then:
\begin{longlist}[(ii)]
\item[(i)]
There exist $\lambda>0$ and a constant $c_m$ such that for all $0\leq
s\leq t\leq1$,
%
\begin{equation}
\label{tight-bound} \sup_{h\in(0,1)}\frac{1}{h^4 \ln(1/h)} \mathbf E \bigl[
\bigl| X^{2m,h}_t-X^{2m,h}_s \bigr|^2
\bigr] \leq c_m \llvert t-s\rrvert ^\lambda.
\end{equation}

\item[(ii)] The family $\{h^{-2} \ln(1/h)^{-1/2} X^{2m,h}; h>0\}$
is tight in $\mathcal C([0,1])$.
\end{longlist}
\end{proposition}

In\vspace*{1pt} order to prove Proposition~\ref{propotight}, recall that
$X^{2m,h}_t=\frac{16}{(2m)!} I_{2m}( (f_h+g_{h,t})\cdot\mathbf
{1}_{[0,t]^{2m}}
)$ with $f_h,g_h$ defined by~(\ref{definf-h})--(\ref{defing-h}). We
will also use the following additional property of $g_h$, which can be
readily checked with the help of~(\ref{bound-Phi-2})--(\ref
{bound-Phi-3}), as in the proof of Lemma~\ref{lemmag-h}.

%
\begin{lemma}\label{lemmag-h-2}
Fix $m\geq1$, and recall that we write $g_{h,t}$ instead of
$g_{h,t,2m}$ for notational sake. Then there exist $\lambda>0$ and a
constant $c_m$ such that for all $0\leq s\leq t\leq1$, one has
\[
\sup_{h\in(0,1)}\|g_{h,t}-g_{h,s}
\|^2_{L^2([0,s]^{2m})} \leq c_m h^4 \llvert
t-s \rrvert ^\lambda.
\]
\end{lemma}

We can now turn to the proof of the main proposition of this section.

\begin{pf*}{Proof of Proposition~\ref{propotight}}
We prove the two claims of the proposition separately:
\begin{longlist}
\item[\textit{Step} 1.] \textit{Proof of assertion} (i).
Let us write
\begin{eqnarray*}
&& X^{2m,h}_t-X^{2m,h}_s
\\
&&\qquad =
c_m \bigl\{ I_n \bigl( (f_h+g_{h,t})
\cdot\{ \mathbf {1}_{[0,t]^{2m}}-\mathbf{1} _{[0,s]^{2m}}\}
\bigr)+I_n \bigl( (g_{h,t}-g_{h,s})\cdot
\mathbf{1}_{[0,s]^{2m}} \bigr) \bigr\}. %
\end{eqnarray*}
The second term of this decomposition can be treated with Lemma~\ref{lemmag-h-2}. As for the first term, we clearly have
\[
\mathbf E \bigl[ \bigl(I_n \bigl( (f_h+g_{h,t})
\cdot\{ \mathbf{1} _{[0,t]^{2m}}-\mathbf{1} _{[0,s]^{2m}}\} \bigr)
\bigr)^{2} \bigr] \le c_m \bigl(A^{h}_{s,t}+B^{h}_{s,t}
\bigr)
\]
with
\begin{eqnarray*}
A^{h}_{s,t} &=& \mathop{\int_{0<t_1< \cdots<t_{2m}}}_{s<t_{2m} <t}
\Phi_h(t_1,t_{2m})^2
\,dt_1 \cdots \,dt_{2m},
\\
B^{h}_{s,t} &=& \mathop{\int_{0<t_1< \cdots<t_{2m}}}_{s<t_{2m} <t}
g_{h,t}(t_1,\ldots,t_{2m})^2
\,dt_1\cdots \,dt_{2m}.
\end{eqnarray*}
We now bound those two terms: first, split up $A^h_{s,t}$ into
$A^h_{s,t}=c_m \{A^{h,1}_{s,t}+A^{h,2}_{s,t}\}$ with
\[
A^{h,1}_{s,t}\equiv\int_{0 <t_1<s<t_{2m} <t}
(t_{2m}-t_1)^{2m-2} \Phi _h(t_1,t_{2m})^2
\,dt_1 \,dt_{2m}
\]
and
\begin{eqnarray*}
A^{h,2}_{s,t}&\equiv& \int_{s<t_1<t_{2m} <t}
(t_{2m}-t_1)^{2m-2} \Phi_h(t_1,t_{2m})^2
\,dt_1 \,dt_{2m}
\\
&=&\int_{0<t_1<t_{2m} <t-s} (t_{2m}-t_1)^{2m-2}
\Phi_h(t_1,t_{2m})^2
\,dt_1 \,dt_{2m}.
\end{eqnarray*}
Then by~(\ref{bound-Phi-1}), one has for any small $\varepsilon>0$,
\begin{eqnarray*}
A^{h,1}_{s,t}&\leq& c_m h^4\int
_{0 <t_1<s<t_{2m} <t} (t_{2m}-t_1)^{-1}
\,dt_1 \,dt_{2m}
\\
&\leq& c_m h^4\int_{0 <t_1<s}
(s-t_1)^{-1+\varepsilon} \,dt_1 \int_{s<t_{2m} <t}
(t_{2m}-t_1)^{-\varepsilon} \,dt_{2m}
\\
& \leq&
c_m h^4 \llvert t-s \rrvert ^{1-\varepsilon}.
\end{eqnarray*}
As far as $A^{h,2}_{s,t}$ is concerned, we can follow the lines of the
proof of Proposition~\ref{proprenorm} and conclude that $\lim\frac
{1}{h^4 \ln(1/h)}A^{h,2}_{s,t}=c_m \llvert t-s\rrvert $ as $h$ tends to zero,
which gives us a proper estimate.

Finally, the bound for $B^{h}_{s,t}$ is easily derived as follows:
first notice that, according to Definition~(\ref{defing-h}) of
$g_{h,t}$, we have
\[
B^{h}_{s,t} \le c \mathop{\int_{0<t_1< \cdots<t_{2m}}}_{s<t_{2m} <t}
\bigl[\Phi_{h}^{2}(t_1,t) +
\Phi_{h}^{2}(0,t) + \Phi_{h}^{2}(0,t_{2m})
\bigr] \,dt_1\cdots \,dt_{2m}.
\]
The three terms above are handled easily, and along the same lines,
thanks to~(\ref{bound-Phi-1}). For the first one, we get, for instance,
\begin{eqnarray*}
&& \mathop{\int_{0<t_1<t_{2m}<t}}_{s<t_{2m} <t} (t_{2m}-t_1)^{2m-2}
\Phi_h^2(t_1,t) \,dt_1\cdots
\,dt_{2m}
\\[-1pt]
&&\qquad \leq c_{m} h^4 \mathop{\int_{0<t_1<t_{2m}<t}}_{s<t_{2m} <t}
(t_{2m}-t_1)^{2m-2} (t-t_1)^{-2m+1}
\,dt_1 \,dt_{2m}
\\[-1pt]
&&\qquad \leq c_m h^4 (t-s) \int_{0< t_1 <s }
(t-t_1)^{-1} \,dt_1
\\[-1pt]
&&\quad\qquad{}+c_m
h^4 \int_{s < t_1 < t}(t-t_1)^{2m-1}
(t-t_1)^{-2m+1} \,dt_1
\\[-1pt]
&&\qquad \leq c_m h^4 \llvert t-s\rrvert ^{1-\varepsilon}
\end{eqnarray*}
for any small $\varepsilon>0$. Gathering now our estimates on
$A^{h}_{s,t}$ and
$B^{h}_{s,t}$, we have proved our claim~(\ref{tight-bound}).
\end{longlist}
\begin{longlist}
\item[\textit{Step} 2.] \textit{Proof of assertion} (ii).
With inequality~(\ref{tight-bound}) in hand, the tightness result is
easily deduced. Indeed, the random variable $X^{2m,h}_t-X^{2m,h}_s$
living in a finite chaos, we are in a position to use
hypercontractivity (see \cite{Nu-bk}) and assert that for all $p\geq1$,
\[
\sup_{h\in(0,1)} \frac{1}{h^{4p} \ln(1/h)^p} \mathbf E \bigl[ \bigl|
X^{2m,h}_t-X^{2m,h}_s \bigr|^{2p}
\bigr] \leq c_{m,p} \llvert t-s \rrvert ^{\lambda p}.
\]
As we have done before, Kolmogorov's tightness criterion is therefore
verified for every $p$ such that $\lambda p>1$, which finishes our proof.\quad\qed
\end{longlist}\noqed
\end{pf*}

\section{$L^2$ modulus of $2$-dimensional local time on chaoses}\label{sec2-dim}
We now carry on the task of proving Theorem~\ref{thmcvgce-L2-modulus-chaos-2d} for projections of the quantity
$H_t^{h}(B)$ defined by~(\ref{eqhamiltonian-mod-cty-loc-time}) when
$B$ is a two-dimensional Brownian motion. For the sake of simplicity,
we shall take up most of the notation introduced at Section~\ref
{sec1-dim}, starting from the fact that our Hamiltonian is written
$H_t^{h}(B)$ independently of the fact that $B$ is a one-dimensional or
a two-dimensional Brownian motion. Like in \cite{HN10}, we shall also
invoke the following important representation formula for $H_t^{h}(B)$:
%
\begin{eqnarray}
\label{eqlocal-modulus-dirac} H_t^{h}(B) &=& \int_{[0,t]^{2}}
\bigl[\delta_0(B_v-B_u+h)
\nonumber\\[-10pt]\\[-10pt]
&&\hspace*{27pt}{} +\delta_0(B_v-B_u-h)-2
\delta_0(B_v-B_u) \bigr] \,du \,dv.\nonumber
\end{eqnarray}

%
\begin{remark}
The reader should be aware of the fact that expression (\ref
{eqlocal-modulus-dirac}) is formal, since the self-intersection local
time is a divergent quantity for a two-dimensional Brownian motion.
Notice, however, that only projections on fixed chaoses will be
considered in the sequel, and all projections of the random
distribution defined by~(\ref{eqlocal-modulus-dirac}) are well defined.
\end{remark}

Next, we introduce the equivalent of the functions $\Phi$ introduced at
Notation~\ref{notdef-simplex-Phi-h}. In the 2-d case, we will let this
set of functions appear in a Fourier transform procedure, as follows.

%
\begin{notation}\label{notdef-simplex-Phi-h-2}
For every $n\geq2$, every $\mathbf{i}=(i_1,\ldots,i_n)\in\{1,2\}^n$ and
every $h\in\mathbb R^2$, we define a function $\Phi_\mathbf
{i}(t_1,t_2)$ as
\[
\Phi_\mathbf{i}(t,s)=\Phi_{\mathbf{i},h}(t,s):=\int
_{\mathbb R^2} d\xi \Biggl( \prod_{k=1}^n
\xi_{i_k} \Biggr) \frac{\{1-\cos
(\langle h,\xi\rangle)\}
}{|\xi
|^4} e^{-(t-s)|\xi|^2/2}.
\]
\end{notation}

%
\begin{remark}
In order to draw a link between $\Phi_\mathbf{i}$ and the function
$\Phi=\Phi^{1\mbox{-}\mathrm{d}}$ introduced at Notation~\ref{notdef-simplex-Phi-h},
observe that, at least for $n=2m$ even (the only cases of interest in
our study), one can also write $\Phi^{1\mbox{-}\mathrm{d}}$ as
\begin{eqnarray*}
\Phi^{1\mbox{-}\mathrm{d}}_{h,2m}(t,s) &=& \int_0^h
p^{(2m-2)}_{t-s}(y) (h-y) \,dy
\\
&=& p^{(2m-2)}_{t-s}(h)-p^{(2m-2)}_{t-s}(0)
\\
&=& \frac{1}2 \bigl\{p^{(2m-2)}_{t-s}(h)+p^{(2m-2)}_{t-s}(-h)-2
p^{(2m-2)}_{t-s}(0) \bigr\}
\\
&=& c \int_{\mathbb R} d\xi \xi^{2m-2} \bigl\{1-\cos(h
\xi) \bigr\} e^{-(t-s)\xi^2 /2},
\end{eqnarray*}
where we have used the Fourier representation $p_{t-s}(x)=c\int_{\mathbb R
}\,d\xi e^{\imath x \xi} e^{-(t-s) \xi^2/2}$.
\end{remark}

The continuity properties of the functions $\Phi_\mathbf{i}$,
mimicking~(\ref{bound-Phi-1})--(\ref{bound-Phi-3}), are summarized below.

%
\begin{lemma}
Fix $m \geq1$ and $\alpha\in(0,2)$. Then there exists a constant
$c_{m,\alpha}$ such that for every $h\in\mathbb R^2$ and all $0\leq
t_1 <s <t<t_2$,
%
\begin{eqnarray}
\label{borne-phi-h-i} \max_{\mathbf{i}\in\{1,2\}^{2m}} \bigl|\Phi_{\mathbf{i},h}(t,s)\bigr| &\leq&
c_{m,\alpha} |h|^\alpha\llvert t-s \rrvert ^{-m+1-(\alpha/2)},\hspace*{-10pt}
\\
\label{borne-phi-h-i-2} \max_{\mathbf{i}\in\{1,2\}^{2m}} \bigl|\Phi_{\mathbf{i},h}(t,t_1)-
\Phi _{\mathbf{i},h}(s,t_1)\bigr| &\leq& c_{m,\alpha}
|h|^\alpha\llvert t-s \rrvert ^\lambda\llvert s-t_1
\rrvert ^{-m+1-(\alpha/2)-\lambda},\hspace*{-10pt}
\\
\label{borne-phi-h-i-3} \qquad\quad\max_{\mathbf{i}\in\{1,2\}^{2m}} \bigl|\Phi_{\mathbf{i},h}(t_2,t)-
\Phi _{\mathbf{i},h}(t_2,s)\bigr| &\leq& c_{m,\alpha}
|h|^\alpha\llvert t-s \rrvert ^\lambda\llvert t_2-t
\rrvert ^{-m+1-(\alpha/2)-\lambda}.\hspace*{-10pt}
\end{eqnarray}
\end{lemma}

The strategy of the proof for Theorem~\ref{thmcvgce-L2-modulus-chaos-2d} is now similar to the one-dimensional
case of Theorem~\ref{thmcvgce-L2-modulus-chaos}: exact computation of
the chaos decomposition, analysis of the variance and contraction
properties for $H_t^{h}(B)$. This is why we shall skip some details
below, and mainly stress the differences between the 1-d and the 2-d case.

\subsection{Stochastic analysis in dimension 2}
The Malliavin calculus setting we shall use in this section is very
similar to the one explained at Section~\ref
{secprelim-stoch-analysis}. However, we stress here some differences
between stochastic analysis for 1-d and 2-d Brownian motions.

Notice first that our standing Wiener space is now the space of
$\mathbb R
^2$-valued continuous functions $\mathcal C(\mathbb R_+; \mathbb
R^2)$, while the related
Hilbert space is $\mathcal H\equiv(L^{2}([0,1]))^{2}$. We set
$B=(B^{1},B^{2})$ for the two-dimensional Wiener process and for
$h=(h^{1},h^{2})\in\mathcal H$ we define $B(h)=B^{1}(h^{1})+B^{2}(h^{2})$.
Starting from this definition of Wiener integral, the Malliavin
derivatives and Sobolev spaces are defined along the same lines as in
Section~\ref{secprelim-stoch-analysis}.

Stroock's formula takes the following form in the two-dimensional
situation: designate by $\mathbf{i}=(i_1,\ldots,i_n)$ a generic
element of $\{
1,2\}^{n}$. Then for a functional $F\in\mathbb D^{\infty,2}$, we have
$J_n(F)=I_{n}(f_{n})$, with
%
\begin{eqnarray}\label{eqstroock-formula-2d}
I_{n}(f_{n}) = \sum
_{\mathbf{i}\in\{1,2\}^{n}} \int_{[0,1]^{n}} f_{\mathbf{i}
}(t_1,
\ldots,t_n) \,dB^{i_1}_{t_1} \cdots
\,dB^{i_n}_{t_n},
\nonumber\\[-10pt]\\[-10pt]
\eqntext{\displaystyle \qquad
f_{\mathbf{i}}(t_1,
\ldots,t_n) = \frac{\mathbf E [\mathcal
D_{t_1,\ldots,t_n}^{\mathbf{i}} F  ]}{n!}.}
\end{eqnarray}

Finally, Propositions~\ref{propfourth-moment} and~\ref{propfourth-moment-d-dim} are still valid in our 2-d Wiener space
context, except for the fact that the expression for the $r$th
contraction ($r\in\{1,\ldots,n\}$) of a given kernel $f$ reads as
follows: for $\mathbf{k}^1,\mathbf{k}^2 \in\{1,2\}^{n-r}$ and
$\mathbf{t}^1,\mathbf{t}^2 \in[0,1]^{n-r}$,
%
\begin{equation}
\label{eqdef-contraction-2d} (f\otimes_r f
)_{(\mathbf{k}^1,\mathbf{k}^2)} \bigl(\mathbf{t}^1,\mathbf {t}^2
\bigr):=\sum_{\mathbf{l}\in\{1,2\}
^r}\int_{[0,1]^r} \,d
\mathbf{s} f_{(\mathbf{k}^1,\mathbf
{l})} \bigl(\mathbf{t}^1,\mathbf{s}
\bigr)f_{(\mathbf{k}^2,\mathbf{l})} \bigl(\mathbf{t} ^2,\mathbf{s} \bigr).
\end{equation}

\subsection{Chaos decomposition of $H_t^{h}(B)$}

We are now ready to compute the projections $X^{n,h}$ of $H^{h}(B)$ on
chaoses, which is the analogous statement to Proposition~\ref{propproj-1-d} in the 2-d situation.

%
\begin{proposition}
For every $n\geq1$ and every nonzero $h\in\mathbb R^2$, recall that
we have
set $X^{n,h}_{t}= J_{n}(H^{h}_{t}(B))$ for the projection of
$H_t^{h}(B)$ onto the $n$th Wiener chaos. Then we have
$X^{n,h}_{t}=0$ if $n$ is odd and
%
\begin{eqnarray}\label{proj-2-dim}
X^{n,h}_{t}=\frac{c}{n!} \sum
_{\mathbf{i}\in\{1,2\}^n}\int_{[0,t]^n} \bigl\{
f_\mathbf{i}(t_1,\ldots,t_n)+g_{\mathbf{i},t}(t_1,
\ldots,t_n) \bigr\} \,dB^{i_1}_{t_1} \cdots
\,dB^{i_n}_{t_n}
\nonumber\\[-12pt]\\[-8pt]
\eqntext{\mbox{if $n$ is even}}
\end{eqnarray}
for some universal constant $c$. In the previous equation, the
symmetric functions $f_\mathbf{i}\in L^2(\mathbb R_+^n)$ and
$g_{\mathbf{i},t}\in
L^2([0,t]^n)$ are defined for each $\mathbf{i}\in\{1,2\}^n$ by
%
\begin{eqnarray}
\label{definf-h-2} && f_\mathbf{i}(t_1,\ldots,t_n)
\nonumber\\[-8pt]\\[-8pt]
&&\qquad =f_{\mathbf{i}}^h(t_1,
\ldots,t_n):=\Phi _\mathbf{i} \bigl(\min (t_1,
\ldots,t_n),\max(t_1,\ldots,t_n) \bigr)\nonumber
\end{eqnarray}
and
\begin{eqnarray}
\label{defing-h-2}
&& g_{\mathbf{i},t}(t_1,\ldots,t_n)\nonumber
\\
&&\qquad =g_{\mathbf{i},t}^h(t_1,
\ldots,t_n)
\\
&&\qquad:=-\Phi_\mathbf{i} \bigl(\min(t_1,\ldots,t_n),t
\bigr)+\Phi_\mathbf {i}(0,t)
-\Phi_\mathbf{i} \bigl(0,\max
(t_1,\ldots,t_n) \bigr),\nonumber
\end{eqnarray}
where we recall that the functions $\Phi_\mathbf{i}$ are introduced at
Notation~\ref{notdef-simplex-Phi-h-2}.
\end{proposition}

\begin{pf}
By applying Stroock's formula~(\ref{eqstroock-formula-2d}) to
expression~(\ref{eqlocal-modulus-dirac}) in a similar manner as in the
proof of Proposition~\ref{propproj-1-d}, we obtain that $X^{n,h}_t$ is
equal to
\begin{eqnarray*}
& =& \frac{c}{n!}\sum_{\mathbf{i}\in\{1,2\}^{n}} \int
_{[0,t]^n} \Biggl(\int_{\mathcal S_t^2} \biggl[
\frac{\partial^n p_{v-u}}{\partial x_\mathbf{i}}(h)
+\frac
{\partial^n
p_{v-u}}{\partial x_\mathbf{i}}(-h)
\\
&&\hspace*{151pt}{}
-2\frac{\partial^n
p_{v-u}}{\partial x_\mathbf{i}
}(0) \biggr]
\\
&&\hspace*{130pt}{}\times
\prod_{i=1}^n \mathbf{1}_{[u,v]}(t_i)
\,du \,dv \Biggr)\,
dB_{t_1}^{i_1} \cdots \,dB_{t_n}^{i_n}
\\
&=& c \sum_{\mathbf{i}\in\{1,2\}^{n}} \int_{\mathcal S_t^n}
\biggl(\int_{0}^{t_1}\!\! \int_{t_n}^{t}
\biggl[\frac{\partial^n p_{v-u}}{\partial x_\mathbf{i}}(h)+\frac
{\partial^n
p_{v-u}}{\partial x_\mathbf{i}}(-h)
\\
&&\hspace*{147pt}{} -2\frac{\partial^n
p_{v-u}}{\partial x_\mathbf{i}
}(0)
\biggr]\,du\,dv \biggr) \,dB_{t_1}^{i_1} \cdots \,dB_{t_n}^{i_n},
\end{eqnarray*}
where $p_t(x)$ stands here for the $2$-dimensional Gaussian kernel and
where we have set $\partial x_\mathbf{i}:=\partial x_{i_1}\cdots
\,\partial
x_{i_n}$. Observe first that $\frac{\partial^n p_{v-u}}{\partial
x_\mathbf{i}
}(h)+\frac{\partial^n p_{v-u}}{\partial x_\mathbf{i}}(-h)-2\frac
{\partial^n
p_{v-u}}{\partial x_\mathbf{i}}(0)$ vanishes when $n$ is odd, which
yields our
claim $X^{n,h}_{t}=0$ in this situation. In the case where $n$ is even,
use the Fourier representation
\[
p_{v-u}(x)=c\int_{\mathbb R^2} e^{\imath\langle x,\xi\rangle}
e^{-(v-u) |\xi
|^2/2} \,d\xi %
\]
and Fubini's theorem in order to derive
\[
X^{n,h}_t=c \sum_{\mathbf{i}\in\{1,2\}^{n}} \int
_{\mathcal S_t^n} \bigl[\Phi_\mathbf{i}(t_n,t_1)-
\Phi_\mathbf{i}(t_1,t)+\Phi _\mathbf{i}(0,t)-
\Phi_\mathbf{i}(t_n,0) \bigr] \,dB_{t_1}^{i_1}
\cdots \,dB_{t_n}^{i_n}.
\]
Formula~(\ref{proj-2-dim}) follows by symmetrization.
\end{pf}

\subsection{Asymptotic behavior of the variance}\label{secbehavior-variance-2d}

With expression~(\ref{proj-2-dim}) in hand, we now proceed as for the
one-dimensional case, and compute $\mathbf E[X^{2m,h}_t X^{2m,h}_s]$ in
order to see how this kind of quantity scales in $h$. Let us first
label the following analytic lemma which will feature in our future
computations.

%
\begin{lemma}\label{leminvar-int}
Fix $m\geq1$ and $\varphi\dvtx  \mathbb R^2 \to\mathbb R$ such that
%
\begin{equation}
\label{conditions-phi} \bigl|\varphi(x,y)\bigr| \leq c \bigl\{|x|^{1-\varepsilon}|y|^{1-\varepsilon
}+|x|^{1+\varepsilon}|y|^{1+\varepsilon
}
\bigr\}
\end{equation}
for some small $\varepsilon>0$. For every nonzero $e\in\mathbb R^2$, set
%
\begin{equation}
\label{eqdef-L-phi}\qquad L^\varphi_{2m,e}:=\int_{\mathbb R^2}
d\xi\int_{\mathbb R^2} d\eta \frac{\langle
\xi,\eta\rangle^{2m}}{|\xi|^4|\eta|^4} \varphi \bigl(
\langle \xi,e\rangle, \langle\eta,e\rangle \bigr) \exp \biggl(-\frac{1}2
\bigl(| \xi|^2+|\eta |^2 \bigr) \biggr).
\end{equation}
Then $L^\varphi_{2m,e}$ is well defined and for all unit vectors
$e,\tilde{e}\in\mathbb R^2$, one has $L^\varphi_{2m,e}=L^\varphi
_{2m,\tilde{e}}$. We denote by $L_{2m}^\varphi$ this common quantity.
\end{lemma}

\begin{pf}
The fact that $L^\varphi_{2m,e}$ is well defined can be easily checked
using~(\ref{conditions-phi}). As for the second assertion, introduce
the rotation $A$ which sends $e$ to $\tilde{e}$ and then use the
isometric change of variables $\xi=A^\ast\tilde{\xi}$, $\eta
=A^\ast
\tilde{\eta}$, so as to turn $L^\varphi_{2m,e}$ into~$L^\varphi
_{2m,\tilde{e}}$.
\end{pf}

We will also make use of the following uniform estimate for $g_{\mathbf{i},t}^h$, which (as in the proof of Lemma~\ref{lemmag-h}) can be easily
derived from the bound~(\ref{borne-phi-h-i}):

%
\begin{lemma}\label{lemget-rid}
Fix $m\geq1$, and recall that for every $t> 0$ and every $\mathbf
{i}\in\{
1,2\}^{2m}$, $g_{\mathbf{i},t}$ is defined by~(\ref{defing-h}). Then there
exist a constant $c_m$ and a small $\varepsilon>0$ such that for
every $h\in
\mathbb R^2$,
\[
\sup_{t\in[0,1],\mathbf{i}\in\{1,2\}^{2m}} \bigl\| g_{\mathbf{i},t}^h \bigr\|
_{L^2([0,t]^{2m})}^2 \leq c_m |h|^{2+\varepsilon}.
\]
\end{lemma}

We can now compute the correct order of $\mathbf E[X^{2m,h}_t X^{2m,h}_s]$
as follows.

%
\begin{proposition}\label{propvar-dim-2}
Fix $m\geq1$. Then for all $0\leq s \leq t\leq1$, it holds that
\[
\lim_{h\to0} |h|^{-2} \mathbf E \bigl[X^{2m,h}_t
X^{2m,h}_s \bigr] =\sigma_m^2 s
\qquad\mbox{with } \sigma_m^2=\frac{c L_{2m}^\varphi}{(2m-2)!},
\]
where $c$ is a universal constant, where $\varphi$ is defined for every
$(x,y) \in\mathbb R^2$ by
\[
\varphi(x,y):=\int_0^\infty\frac{du}{u^3}
\bigl\{1-\cos(ux) \bigr\} \bigl\{1-\cos (uy) \bigr\}
\]
and where we recall that $L_{2m}^\varphi$ has been introduced at
relation~(\ref{eqdef-L-phi}).
\end{proposition}

\begin{pf}
Recall that
\[
\mathbf E \bigl[X^{2m,h}_t X^{2m,h}_s
\bigr] =\frac{c}{(2m)!} \sum_{\mathbf{i}\in\{
1,2\}^{2m}} \bigl
\langle(f_\mathbf{i}+g_{\mathbf{i},t}) \cdot\mathbf{1}
_{[0,t]^{2m}},(f_\mathbf{i} +g_{\mathbf{i},s}) \cdot
\mathbf{1}_{[0,s]^{2m}} \bigr\rangle_{L^2(\mathbb R^{2m})}
\]
and thanks to Lemma~\ref{lemget-rid}, we only have to focus on the sum
of the terms
\[
\mathcal{A}^\mathbf{i}_{s,t}:=\langle f_\mathbf{i}
\cdot\mathbf{1} _{[0,t]^{2m}},f_\mathbf{i}\cdot \mathbf{1}_{[0,s]^{2m}}
\rangle_{L^2(\mathbb R^{2m})}.
\]
An integration over the simplex gives
\[
\mathcal{A}^\mathbf{i}_{s,t}=\frac{(2m)!}{(2m-2)!} \int
_{\Delta
_s^2} \Phi_{\mathbf{i},h}(t_{2m},t_1)^2
(t_{2m}-t_1)^{2m-2} \,dt_1\,dt_{2m}.
\]
The change of variables $t_{2m}-t_1=\tau$ and $t_1=\sigma$ easily
leads us to
\[
\mathcal{A}_{s,t}^\mathbf{i}=\frac{(2m)!}{(2m-2)!} \int
_0^s (s-\tau ) \Phi_{\mathbf{i}, h/{\tau^{1/2}}}(1,0)^2
\,d\tau.
\]
Setting $e_{h}\equiv\frac{h}{|h|}$, the change of variable
$u=|h|/\tau
^{1/2}$ now gives
\[
\mathcal{A}_{s,t}^\mathbf{i}=\frac{2(2m)!}{(2m-2)!}|h|^2
\int_{|h|/s^{1/2}}^\infty u^{-3} \biggl( s-
\frac{|h|^2}{u^2} \biggr)\Phi _{\mathbf{i},u e_{h}}(1,0)^2 \,du.
\]
By using~(\ref{borne-phi-h-i}), one can check that $|h|^2 \int_{|h|/s^{1/2}}^\infty u^{-5}\Phi_{\mathbf{i},u e_{h}}(1,0)^2 \,du
\to0$ as
$h\to0$, so that the main contribution will come from the terms
\[
\widehat{\mathcal{A}}^\mathbf{i}_{s,t}:=\frac{2(2m)!}{(2m-2)!}|h|^2
s \int_{|h|/s^{1/2}}^\infty u^{-3}
\Phi_{\mathbf{i},u e_{h}}(1,0)^2 \,du. %
\]
Now write
\begin{eqnarray*}
&& \Phi_{\mathbf{i},u e_{h}}(1,0)^2
\\
&&\qquad =\int_{\mathbb R^2}d\xi\int
_{\mathbb R^2} d\eta \frac{
\prod_{k=1}^{2m} \xi_{i_k} \eta_{i_k} }{|\xi|^4 |\eta
|^4}
\\
&&\hspace*{91pt}{}\times \bigl\{ 1-\cos \bigl(u\langle
e_{h},\xi\rangle \bigr) \bigr\}\bigl\{ 1-\cos \bigl(u\langle
e_{h},\eta \rangle \bigr) \bigr\}e^{-(1/2)(|\xi|^2+|\eta|^2)}
\end{eqnarray*}
and observe that $\sum_{\mathbf{i}\in\{1,2\}^{2m}} \prod_{k=1}^{2m}
\xi_{i_k} \eta_{i_k} =\langle\xi,\eta\rangle^{2m}$. Thus, thanks to
Lemma~\ref{leminvar-int} and using Fubini theorem, we deduce that
\[
|h|^{-2}\sum_{\mathbf{i}\in\{1,2\}^{2m}}\widehat{\mathcal
{A}}^\mathbf {i}_{s,t}=\frac
{2(2m)!}{(2m-2)!}s\cdot
L_{2m}^{\varphi_h},
\]
where $\varphi_h$ is defined as
\[
\varphi_h(x,y):=\int_{|h|/s^{1/2}}^\infty
\frac{\{1-\cos(ux)\}\{
1-\cos
(uy)\}}{u^3}\,du. %
\]
Finally, the convergence of $L_{2m}^{\varphi_h}$ toward
$L_{2m}^{\varphi
}$ easily follows from the fact that $\varphi$ satisfies relation
(\ref
{conditions-phi}) for some small $\varepsilon>0$, and this achieves
the proof.
\end{pf}

\subsection{Contractions}

We now turn to the contractions estimation for the functions
$f_{h},g_{h}$, where our two-dimensional contractions are defined by
(\ref{eqdef-contraction-2d}). The following is of course an analog of
Proposition~\ref{propcontraction-f-h-plus-g-h} in our 2-d setting.

%
\begin{proposition}\label{propcontractions-2-d}
For every $r\in\{1,\ldots,n-1\}$, one has
\begin{eqnarray*}
&& \bigl\| (f_{.} \mathbf{1}_{[0,t]^n}+g_{.,t}) \otimes
_r (f_{.} \mathbf{1}
_{[0,t]^n}+g_{.,t})\bigr\|^2
\\
&&\qquad = \sum_{\mathbf{i}\in\{1,2\}^{2n-2r}} \bigl\| \bigl((f_{.}
\mathbf{1} _{[0,t]^n}+g_{.,t})\otimes
_r (f_{.} \mathbf{1}_{[0,t]^n}+g_{.,t})
\bigr)_{\mathbf{i}}\bigr\| ^2_{L^2([0,t]^{2n-2r})}
\\
&&\qquad =o \bigl(|h|^4
\bigr).
\end{eqnarray*}
\end{proposition}

\begin{pf}
Thanks to Lemma~\ref{lemget-rid}, it suffices to focus on the sum
\[
\sum_{\mathbf{k}^1,\mathbf{k}^2\in\{1,2\}^{n-r}} \bigl\| \bigl((f_{.} \mathbf{1}
_{[0,t]^n})\otimes_r (f_{.}
\mathbf{1}_{[0,t]^n}) \bigr)_{(\mathbf{k}^1,\mathbf{k}^2)}\bigr\| ^2_{L^2([0,t]^{2n-2r})}.
\]
Assume first that $2m\geq4$ and $2 \leq r \leq2m-2$. Then we can
follow the lines of the proof of Proposition~\ref{propcontraction-f-h-plus-g-h} and deduce that
\begin{eqnarray*}
&& \sum_{\mathbf{k}^1,\mathbf{k}^2\in\{1,2\}^{n-r}} \bigl\| \bigl((f_{.} \mathbf{1}
_{[0,t]^n})\otimes_r (f_{.}
\mathbf{1}_{[0,t]^n}) \bigr)_{(\mathbf
{k}^1,\mathbf
{k}^2)}\bigr\| ^2_{L^2([0,t]^{2n-2r})}
\\
&&\qquad = c_m \mathop{\sum_{\mathbf{k}^1,\mathbf{k}^2 \in\{1,2\}
^{n-r}}}_{\mathbf{l}^1,\mathbf{l}^2 \in
\{1,2\}^r}
\int_{(\mathcal S_t^2)^2} \int_{(\mathcal S_t^2)^2} \prod
_{i,j=1}^2 \Phi _{(\mathbf{k}^i,\mathbf{l}^j),h} \bigl(\max \bigl(
\sigma_2^i,\tau_2^j \bigr),\min
\bigl(\sigma_1^i,\tau_1^j \bigr)
\bigr)
\\[-2pt]
&&\quad\qquad\hspace*{120pt}{} \times\prod_{k=1}^2 \bigl(
\sigma_2^k-\sigma_1^k
\bigr)^{r-2}
\\
&&\hspace*{178pt}{} \times \bigl(\tau _2^k-
\tau_1^k \bigr)^{n-r-2} \,d\sigma_1^k
\,d\sigma_2^k \,d\tau_1^k \,d
\tau_2^k.
\end{eqnarray*}
Now, plugging the bound~(\ref{borne-phi-h-i}) (uniformly over
$(\mathbf{k}
^i,\mathbf{l}^j)$) into the latter expression yields, similar to
(\ref{eqbnd-R-n-r}): for any small $\varepsilon>0$,
\[
\sum_{\mathbf{k}^1,\mathbf{k}^2\in\{1,2\}^{n-r}} \bigl\| \bigl((f_{.} \mathbf{1}
_{[0,t]^n})\otimes_r (f_{.}
\mathbf{1}_{[0,t]^n}) \bigr)_{(\mathbf{k}^1,\mathbf{k}^2)}\bigr\| ^2_{L^2([0,t]^{2n-2r})}
\leq c_m |h|^{4+4\varepsilon} \mathcal J_\varepsilon
\]
with
\begin{eqnarray*}
\mathcal J_\varepsilon &:=& \int_{(\Delta_1^2)^2} \int
_{(\Delta
_1^2)^2} \prod_{i,j=1}^2
\bigl(\max \bigl(\sigma_2^i,\tau_2^j
\bigr)-\min \bigl(\sigma_1^i,\tau_1^j
\bigr) \bigr)^{-3/2-\varepsilon/2}
\\[-2pt]
&&\hspace*{74pt}{}\times \prod_{k=1}^2
\,d\sigma_1^k \,d\sigma_2^k\, d
\tau_1^k \,d\tau_2^k.
\end{eqnarray*}
By Lemma~\ref{lemtechni}, we know that this integral is finite for
$\varepsilon>0$ small enough, which achieves the proof of the
proposition in
the case ($2m\geq4$, $2 \leq r \leq2m-2$).

The two situations ($2m\geq4$, $r\in\{1,2m-1\}$) and ($2m=2$, $r=1$)
can also be handled with the same arguments as in the proof of
Proposition~\ref{propcontraction-f-h-plus-g-h} (with the help of Lemma
\ref{lemtechni} as well). Details are left to the reader.
\end{pf}

As in Section~\ref{seccontractions-1-d}, by combining Propositions
\ref{propvar-dim-2} and~\ref{propcontractions-2-d} we end up with the
following convergence in law result for the finite-dimensional
distributions of~$X^{2m,h}$:

%
\begin{proposition}\label{propresume}
Taking up the above notation, consider $t_1,\ldots,t_d\in[0,1]$ and
$m\ge
1$. Then as $h\to0$, we have
\[
\frac{1}{|h|} \bigl(X^{2m,h}_{t_1},\ldots,X^{2m,h}_{t_d}
\bigr) \xrightarrow{(d)} \sigma_{m} \mathcal N(0,\Gamma) \qquad
\mbox{where } \sigma_{m}^{2}= \frac{c L_{2m}^\varphi}{(2m-2)!}
\]
and $\mathcal N(0,\Gamma)$ is the centered Gaussian law in $\mathbb
R^d$ with
covariance matrix $\Gamma(i,j)=\min(t_i,t_j)$. Recall that the quantity
$L_{2m}^\varphi$ has been defined in Proposition~\ref{propvar-dim-2}.
\end{proposition}

Let us briefly check point \textup{(iii)} of Theorem~\ref{thmcvgce-L2-modulus-chaos-2d}, that is, the divergence of the series
of variances, as it is less obvious than in the 1-d case.

%
\begin{proposition}
With the notation of Proposition~\ref{propresume}, it holds that
$\sum_{m=1}^\infty\sigma_m^2 =\infty$.
\end{proposition}

\begin{pf} One has
%
\begin{eqnarray}\label{minor}
\qquad\sum_{m=1}^\infty\sigma_m^2
&=& c \sum_{m=1}^\infty\frac
{L_{2m}^\varphi
}{(2m-2)!}\nonumber
\\
&=& c \int_{\mathbb R^2} d\xi\int_{\mathbb R^2} d\eta
\frac{\langle\xi,\eta
\rangle^2}{|\xi|^4 |\eta|^4} \bigl\{e^{\langle\xi,\eta\rangle
}+e^{-\langle
\xi,\eta\rangle} \bigr\}\varphi(
\xi_1,\eta_1) e^{-(1/2) (|\xi
|^2+|\eta|^2)}
\nonumber\\[-8pt]\\[-8pt]
&\geq& c \int_{\mathbb R^2} d\xi\int_{\mathbb R^2} d\eta
\frac
{\langle\xi,\eta
\rangle^2}{|\xi|^4 |\eta|^4} \varphi(\xi_1,\eta_1)
e^{-(1/2)|\xi-\eta|^2}\nonumber
\\
&\geq& c \int_{[R,\infty)^2} d\xi\int_{B_\xi} d\eta
\frac{\langle\xi,\eta\rangle^2}{|\xi|^4 |\eta|^4} \varphi(\xi_1,\eta_1) \nonumber
\end{eqnarray}
for every $R>0$ and where the notation $B_\xi$ refers to the unit ball
around $\xi$. Now observe that for $R$ large enough, $\xi\in
[R,\infty
)^2$ and $\eta\in B_\xi$, one has
\[
\varphi(\xi_1,\eta_1)=\xi_1^2
\cdot\varphi \biggl(1,\frac{\eta
_1}{\xi_1} \biggr) \geq\xi_1^2
\cdot c_\varphi\qquad\mbox{with } c_\varphi=\inf
_{1/2 \leq x \leq2} \varphi(1,x) >0
\]
and also $\langle\xi,\eta\rangle^2 \geq\frac{1}2 |\xi|^4$. Therefore,
going back to~(\ref{minor}), one has for $R$ large enough and a
suitable (finite) $\widetilde{R}$,
\[
\sum_{m=1}^\infty\sigma_m^2
\geq c \int_{[R,\infty)^2} \frac{\xi
^2}{|\xi
|^4} \,d\xi \geq c \int
_{\widetilde{R}} \frac{dr}{r},
\]
which achieves the proof.
\end{pf}

\subsection{Tightness}
In order to complete the proof of Theorem~\ref{thmcvgce-L2-modulus-chaos-2d}, we are now left with the tightness
property for the family of processes $\{h^{-1} X^{2m,h}_t, t\in
[0,1]\}$. The following proposition is thus the equivalent of
Proposition~\ref{propotight} in our 2-d context.

%
\begin{proposition}\label{propotight-2}
Fix $m\geq1$. Then:
\begin{longlist}[(ii)]
\item[(i)]
There exist $\lambda>0$ and a constant $c_m$ such that for all $0\leq
s\leq t\leq1$,
%
\begin{equation}
\label{tight-bound-2} \sup_{ |h|\in(0,1)}\frac{1}{|h|^2} \mathbf E \bigl[
\bigl| X^{2m,h}_t-X^{2m,h}_s \bigr|^2
\bigr] \leq c_m \llvert t-s\rrvert ^\lambda.
\end{equation}

\item[(ii)] The family $\{X^{2m,h}; |h|>0\}$ is tight in $\mathcal
C([0,1])$.
\end{longlist}
\end{proposition}

\begin{pf}
We use the same arguments as in the proof of Proposition~\ref{propotight}. First, observe that
\[
\mathbf E \bigl[ \bigl| X^{2m,h}_t-X^{2m,h}_s
\bigr|^2 \bigr] \leq c_m \sum_{\mathbf{i}\in\{
1,2\}^{2m}}
\bigl\{ A^{\mathbf{i}}_{s,t}+B^\mathbf{i}_{s,t}+\|
g_{\mathbf{i},t}-g_{\mathbf{i},s}\| _{L^2([0,s]^{2m})}^2 \bigr\},
\]
where
\begin{eqnarray*}
A^\mathbf{i}_{s,t}&:=&\mathop{\int_{0<t_1 <\cdots<t_{2m}}}_{
s<t_{2m} <t}\Phi _\mathbf{i}(t_1,t_{2m})^2
\,dt_1 \cdots \,dt_{2m},
\\
B^\mathbf{i}_{s,t}&:=&\mathop{\int_{0<t_1 <\cdots<t_{2m}}}_{
s<t_{2m} <t}g_{\mathbf{i},t}(t_1,\ldots,t_{2m})^2
\,dt_1 \cdots \,dt_{2m}.
\end{eqnarray*}
By using both~(\ref{borne-phi-h-i-2}) and~(\ref{borne-phi-h-i-3}), it
is readily checked that
\[
\max_{\mathbf{i}\in\{1,2\}^{2m}}|h|^{-2} \|g_{\mathbf
{i},t}-g_{\mathbf{i},s}
\| _{L^2([0,s]^{2m})}^2 \leq c \llvert t-s \rrvert ^\lambda
\]
for some $\lambda>0$. Then the treatments of $\sum_{\mathbf{i}\in\{
1,2\}
^{2m}}A^\mathbf{i}_{s,t}$ and $\sum_{\mathbf{i}\in\{1,2\}
^{2m}}B^\mathbf{i}_{s,t}$, as well
as the derivation of assertion \textup{(ii)}, follow the lines of the
proof of Proposition~\ref{propotight}. For the sake of conciseness, we
do not repeat the details of the procedure.
\end{pf}

\begin{appendix}
\section*{Appendix: A technical lemma}
It only remains to prove the technical result on which the contraction
computations of Propositions~\ref{propcontraction-f-h-plus-g-h} and
\ref{propcontractions-2-d} rely.

%
\begin{lemma}\label{lemtechni}
The three following integrals
%
\begin{eqnarray}
\label{int-sing-1} \qquad &\displaystyle \int_{[0,1]^2} \int_{[0,1]^2}
\prod_{i,j=1}^2 \bigl(\max \bigl(\sigma
^i,\tau ^j \bigr)-\min \bigl(\sigma^i,
\tau^j \bigr) \bigr)^{-3\delta} \,d\sigma^1 \,d
\sigma^2 \,d\tau^1 \,d\tau^2,&
\\
\label{int-sing-2} &\displaystyle \int_{(\Delta_1^2)^2}\int_{[0,1]^2}
\prod_{i,j=1}^2 \bigl(\max \bigl(
\sigma^i,\tau _2^j \bigr)-\min \bigl(
\sigma^i,\tau_1^j \bigr)
\bigr)^{-5\delta} \,d\sigma^1 \,d\sigma ^2\prod
_{k=1}^2 \,d\tau_1^k\, d
\tau_2^k&
\end{eqnarray}
and
%
\begin{equation}
\label{int-sing-3} \qquad\int_{(\Delta_1^2)^2} \int_{(\Delta_1^2)^2}
\prod_{i,j=1}^2 \bigl(\max \bigl(
\sigma_2^i,\tau_2^j \bigr)-\min
\bigl(\sigma_1^i,\tau_1^j \bigr)
\bigr)^{-7\delta} \prod_{k=1}^2 \,d
\sigma_1^k \,d\sigma_2^k\, d
\tau_1^k \,d\tau_2^k
\end{equation}
are convergent if and only if $\delta<1/4$.
\end{lemma}

We only focus on~(\ref{int-sing-3}), since~(\ref{int-sing-1}) and
(\ref
{int-sing-2}) can be treated with similar arguments (see Remark~\ref{rkother-cases} at the end of the proof). In order to ease notation,
we shall also change our time indices and set $(\sigma_{1}^{1},\sigma
_2^1)=(x_1,x_5)$, $(\sigma_{1}^{2},\sigma_2^2)=(x_{2},x_6)$, $(\tau
_{1}^{1},\tau^1_2)=(x_{3},x_7)$, $(\tau_{1}^{2},\tau_2^2)=(x_{4},x_8)$.
Our integral of interest can thus be written as
\begin{eqnarray}
\label{integral} I_\alpha&:=& \int_D
\bigl[(x_7\vee x_5) -(x_3 \wedge
x_1) \bigr]^{-\alpha} \bigl[(x_8\vee
x_5) -(x_4 \wedge x_1) \bigr]^{-\alpha}
\nonumber\\[-8pt]\\[-8pt]
&&\hspace*{10pt}{}\times \bigl[(x_7\vee x_6) -(x_3 \wedge
x_2) \bigr]^{-\alpha} \bigl[(x_8\vee
x_6) -(x_4 \wedge x_2)
\bigr]^{-\alpha} \,dx,\nonumber
\end{eqnarray}
where $D=\{x\in[0,1]^8\dvtx  x_i < x_{4+i}, 1\le i\le4\}$ and $\alpha<7/4$.

The necessity of the condition $\alpha<7/4$ for the convergence of
(\ref{integral}) stems from the following fact: observe that if
\[
\mathcal S:= \bigl\{x\in[0,1]^4\dvtx 0<x_1<x_5<x_2<x_6<x_3<x_7<x_4<x_8<1
\bigr\}
\]
one has
\begin{eqnarray*}
I_{7/4} & \geq& \int
_{\mathcal S} (x_7-x_1)^{-7/4}(x_8-x_1)^{-7/4}(x_7-x_2)^{-7/4}(x_8-x_2)^{-7/4}
\,dx\,dy
\\
& \geq& c \int_{[0,1]^3} (u_1+u_2)^{-7/4}(u_1+u_2+u_3)^{-7/4}
\\
&&\hspace*{31pt}{}\times u_2^{-7/4}
(u_2+u_3)^{-7/4} u_1u_2^2
u_3 \,du_1\,du_2\,du_3
\\
& \geq& c \int_0^1 \frac{dr}{r}
\end{eqnarray*}
by using spherical coordinates.

In order to prove the convergence of $I_\alpha$ when $\alpha<7/4$, we
propose to rely on some block-type representation of the integral,
described as follows. First, given $x\in D$, denote $J_1:=[x_3 \wedge
x_1,x_7\vee x_5]$, $J_2:=[x_4 \wedge x_1,x_8\vee x_5]$, $J_3:=[x_3
\wedge x_2,x_7\vee x_6]$, $J_4:=[x_4 \wedge x_2,x_8\vee x_6]$, so that
\[
I_\alpha=\int_D \prod
_{i=1}^4 \ell(J_i)^{-\alpha}
\qquad\mbox {where } \ell \bigl([a,b] \bigr)=b-a.
\]
Now and for the rest of the proof, we fix a generic permutation $\sigma
\in
\mathfrak{S}_8$ and consider the simplex $\mathcal S^\sigma$
generated by $\sigma$,
that is, $\mathcal S^\sigma:= \{x\in[0,1]^8\dvtx  x_{\sigma(1)} < \cdots<
x_{\sigma(8)}\}$,
assuming that $\mathcal S^\sigma\subset D$. If $J_i=[x_{\sigma
(m_i)},x_{\sigma(n_i)}]$
on $\mathcal S^\sigma$ (for $m_i<n_i \in\{1,\ldots,8\}$ depending on
$\sigma$ as
well), we introduce the block $B_i^\sigma:=\{m_i,m_i+1,\ldots,n_i\}$ and
set $\mathcal B^\sigma:=\{B^\sigma_1,\ldots,B^\sigma_4\}$. Then,
using an elementary
change of variables, it is readily checked that
\[
I_{\alpha,\sigma}:=\int_{\mathcal S^\sigma} \prod
_{i=1}^4 \ell (J_i)^{-\alpha}=
\int_{\mathcal S^\sigma
} \prod_{i=1}^4
(x_{\sigma(n_i)}-x_{\sigma(m_i)})^{-\alpha
}=I_{\alpha,\mathcal B^\sigma},
\]
where we have used the following general notation:

%
\begin{notation}
Given $B_i:=\{m_i,m_i+1,\ldots,n_i\}$ ($i=1,\ldots,4$) with $m_i<n_i
\in\{1,\ldots,8\}$ and $\mathcal B:=\{B_1,\ldots,B_4\}$, we set
\[
I_{\alpha,\mathcal B}:=\int_{0<x_1 <\cdots<x_8 <1} \prod
_{i=1}^4 (x_{n_i}-x_{m_i})^{-\alpha}
\qquad\in[0,\infty].
\]
\end{notation}

Of course, $I_\alpha=\sum_{\sigma\dvtx  \mathcal S^\sigma\subset D}
I_{\alpha,\sigma}=\sum_{\sigma\dvtx  \mathcal S^\sigma\subset D} I_{\alpha,\mathcal B^\sigma}$. Our
key argument to prove
that $I_{\alpha,\mathcal B^\sigma} < \infty$ for every $\sigma\in
\mathfrak{S}_8$ and
$\alpha<\frac{7}4$ lies in the following three basic observations regarding
the four blocks $B^\sigma_i$ composing $\mathcal B^\sigma$:
\begin{longlist}[(iii)]
\item[(i)] $\operatorname{Card}(B^\sigma_i) \geq4$ ($J_i$ involves
the min/max
over four points);

\item[(ii)] $\operatorname{Card}(B^\sigma_i \cup B^\sigma_j ) \geq
6$ if
$i\neq j$
($J_i \cup J_j$ involves the min/max over at least six points);

\item[(iii)] Each of the extremum points $1$ and $8$ appears exactly
twice in $\mathcal B^\sigma$. Indeed, on $\mathcal S^\sigma$, the
minimum $x_{\sigma(1)}$
[resp., maximum $x_{\sigma(8)}$] appears exactly twice as a left (resp.,
right) bound in $J_1,\ldots,J_4$.

%
\begin{figure}[t]

\includegraphics{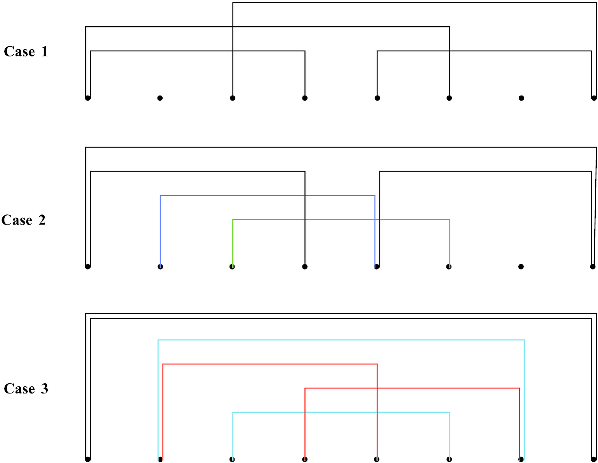}

%
%
%
%
%
%
%
%
%
%
%
\caption{Representation of the ``extremal'' situations in each case, that is,
the $\mathcal B_k$ ($k\in\{0,\ldots,4\}$). Each line connects the extremities
of a block in $\mathcal B_k$. In case~2 (resp., case~3), the black lines are
the ones common to $\mathcal B_1$ and $\mathcal B_2$ (resp., $\mathcal
B_3$ and $\mathcal B_3$).}\label{picconstr}
\end{figure}

Let us now discriminate the possible situations for $\mathcal B^\sigma
$ according
to this last condition \textup{(iii)} (see Figure~\ref{picconstr} for a representation
in each case):
\end{longlist}
\begin{longlist}
\item[\textit{Case} 1.] $1$ and $8$ never appear in the same block
$B_1^\sigma,\ldots,B_4^\sigma$. Then, by focusing on the possibilities for the two
blocks with left-hand side $1$ (resp., the two blocks with right-hand
side $8$), and given the above constraints \textup{(i)--(ii)}, we end up
with $I_{\alpha,\mathcal B^\sigma} \leq I_{\alpha,\mathcal B_0}$ where
\[
\mathcal B_0:= \bigl\{\{1,\ldots,4\},\{1,\ldots,6\},\{5,\ldots,8
\}, \{ 3,\ldots,8\} \bigr\}.
\]

\item[\textit{Case} 2.] $1$ and $8$ appear once and only once in a same block
(and so each of them appears once ``alone'' in another block). Then it
remains to pick one block over the points $\{2,\ldots,7\}$, and given
the constraints \textup{(i)--(ii)} on this block, we can easily conclude
that there exists $k\in\{1,2\}$ such that $I_{\alpha,\mathcal
B^\sigma}\leq I_{\alpha,\mathcal B_k}$ where
\begin{eqnarray*}
{\mathcal B_1}&:=& \bigl\{\{1,\ldots,8\},\{1,\ldots,4\},\{ 5,\ldots,8
\},\{2,\ldots,5\} \bigr\},
\\
{\mathcal B_2}&:=& \bigl\{\{1,\ldots,8\},\{1,\ldots,4\},\{ 5,\ldots,8
\},\{ 3,\ldots,6\} \bigr\}.
\end{eqnarray*}

\item[\textit{Case} 3.] $1$ and $8$ appear twice in a same block (necessarily
$\{1,\ldots,8\}$). Then we have to pick two blocks over the points $\{
2,\ldots,7\}$, and given the constraints \textup{(i)--(ii)} on these two
blocks [note, e.g., that, given \textup{(ii)}, $2$ and $7$ are
necessarily involved in the union of these blocks], we can easily
conclude that there exists $k\in\{1,2\}$ such that $I_{\alpha,\mathcal B^\sigma}
\leq I_{\alpha,\mathcal B_{2+k}}$, where
\begin{eqnarray*}
{\mathcal B_{3}}&:=& \bigl\{\{1,\ldots,8\},\{1,\ldots,8\},\{ 2,
\ldots,5 \},\{4,\ldots,7\} \bigr\},
\\
{\mathcal B_{4}}&:=& \bigl\{\{1,\ldots,8\},\{1,\ldots,8\},\{ 2,\ldots,7
\},\{3,\ldots,6\} \bigr\}.
\end{eqnarray*}
As a consequence of this reasoning, the problem is now reduced to the
sole consideration of the five ``extremal'' integrals $I_{\alpha,\mathcal B_k}$
($k\in\{0,\ldots,4\}$), which can be very easily done with basic
estimates. For instance, if $\alpha=\frac{7}4-\varepsilon$ with
$\varepsilon>0$, one has
\begin{eqnarray*}
I_{\alpha,\mathcal B_0} &=& \int_{0<x_1 < \cdots<x_8 <1} \,dx (x_4-x_1)^{-\alpha
}(x_6-x_1)^{-\alpha}(x_8-x_5)^{-\alpha}(x_8-x_3)^{-\alpha}
\\
&=&c \int_{[0,1]^5} \,du (u_1+u_2)^{-\alpha}(u_1+
\cdots +u_4)^{-\alpha}
\\
&&\hspace*{33pt}{}\times (u_4+u_5)^{-\alpha}(u_2+
\cdots+u_5)^{-\alpha}u_1u_5
\\
&\leq& c \int_{[0,1]^5} \,du u_1^{-1+\varepsilon}
u_5^{-1+\varepsilon} u_2^{-1+(2/3) \varepsilon}u_3^{-1+(2/3) \varepsilon}
u_4^{-1+(2/3)\varepsilon} < \infty,
\end{eqnarray*}
where we have used the elementary bounds
\begin{eqnarray*}
(u_1+u_2)^{-\alpha} &\leq& u_1^{-\alpha},
\qquad(u_4+u_5)^{-\alpha} \leq u_5^{-\alpha},
\\
(u_1+\cdots+u_4)^{-\alpha} &\leq&
u_1^{-\alpha
+3\kappa} u_2^{-\kappa}
u_3^{-\kappa} u_4^{-\kappa}
\end{eqnarray*}
%
with $\kappa:=\frac{1}2-\frac{\varepsilon}{3}$.
\end{longlist}

%
\begin{remark}\label{rkother-cases}
This reduction of the problem, based on a block representation of the
integral, can be easily adapted to prove the convergence of (\ref
{int-sing-1}) [resp.,~(\ref{int-sing-2})], by working with blocks $\{
1,\ldots,4\}$ (resp., $\{1,\ldots,6\}$) made of at least two (resp.,
three) elements. Thus, for relation~(\ref{int-sing-1}) [resp.,~(\ref{int-sing-2})], one can check that the situation reduces to the sole
consideration of two (resp., three) easy-to-handle integrals on specific
simplexes.
\end{remark}
\end{appendix}

\section*{Acknowledgements}
We are grateful to the Associate Editor and to
an anonymous referee for stimulating questions and comments, which
helped us to clarify several crucial points of the paper.


%

\printaddresses

\end{document}